\documentclass{amsart}
\usepackage{graphicx}
\usepackage{amssymb,amscd,amsthm,amsxtra}
\usepackage{latexsym}
\usepackage{epsfig,esint}
\newtheorem{thm}{Theorem}[section]
\newtheorem{cor}[thm]{Corollary}
\newtheorem{lem}[thm]{Lemma}
\newtheorem{prop}[thm]{Proposition}
\theoremstyle{definition}
\newtheorem{defn}[thm]{Definition}
\theoremstyle{remark}
\newtheorem{rem}[thm]{Remark}
\numberwithin{equation}{section}
\newcommand{\R}{\mathbb R}

\newcommand{\be}{\begin{equation}}
\newcommand{\ee}{\end{equation}}

\newcommand{\eps}{\epsilon}

\newcommand{\p}{\partial}

\newcommand{\comment}[1]{}
\begin{document}

\title{The two membranes problem for different operators}

\author{L. Caffarelli}
\address{Department of Mathematics, University of Texas at Austin, Austin, TX 78712, USA}\email{\tt caffarel@math.utexas.edu}
\author{D. De Silva}
\address{Department of Mathematics, Barnard College, Columbia University, New York, NY 10027, USA}
\email{\tt  desilva@math.columbia.edu}
\author{O. Savin}
\address{Department of Mathematics, Columbia University, New York, NY 10027, USA}\email{\tt  savin@math.columbia.edu}
\thanks{D.~D. is supported by NSF grant DMS-1301535.  O.~S.~is supported by  NSF grant DMS-1200701.}

\begin{abstract} We study the two membranes problem for different operators, possibly nonlocal. We prove a general result about the H\"older continuity of the solutions and we develop a viscosity solution approach to this problem. Then we obtain $C^{1,\gamma}$ regularity of the solutions provided that the orders of the two operators are different. In the special case when one operator coincides with the fractional Laplacian, we obtain the optimal regularity and a characterization of the free boundary. \end{abstract}

\maketitle
\section{Introduction}

In this paper we study the two membranes problem for different operators. Physically the problem consists in having two elastic membranes made of possibly different composite materials that are constrained one on top of the other. This is a double obstacle problem in which each membrane can be viewed as the obstacle for the other membrane, and the two obstacles interact at the same time. 

 The two membranes problem for the Laplacian was first considered by Vergara-Caffarelli \cite{VC} in the context of variational inequalities. In this case the situation can be reduced to the classical obstacle problem by looking at the vertical distance between the membranes. The two membranes problem for a nonlinear operator was studied by Silvestre \cite{S1}. He obtained the optimal $C^{1,1}$ regularity of the solutions together with a characterization of the regularity of the free boundary of the coincidence set. The key step is to show that the difference between the two solutions solves an obstacle problem for the linearized operator. 
 
 We also mention that a more general version of the two membranes problem involving $N$ membranes was considered by several authors (see for example \cite{ARS, CCVC, CV}).
 
The two membranes problem for different operators is more challenging mathematically. In the unconstrained parts the membranes solve different equations and therefore their difference solves a fourth order equation rather than a second order equation. For example even in the simplest case of two dimensions and two linear operators, say $\triangle$ and $\tilde \triangle:= \p_{xx}+2\p_{yy}$, the optimal regularity of the solutions seems to be a difficult problem. 

In this paper we consider the two membranes problem for  the large class of elliptic operators, possibly nonlocal, of order $2s \in (0,2]$. The interest in the nonlocal case comes from the applications. It is well known for example that the classical Signorini problem in elasticity which consists in finding the equilibrium position of an elastic body resting on a rigid surface, is modeled by an obstacle problem for the fractional Laplacian $\Delta^{1/2}$. In the case when the elastic body presses against a membrane, one obtains a two membranes obstacle problem involving a fractional Laplacian and a second order operator.  

In the general case, we prove a result about the H\"older continuity of the solutions and we develop a viscosity solution approach. Then we obtain better regularity properties of the solutions provided that the orders of the two operators are different. Heuristically this situation corresponds to the case when one membrane, say the lower membrane, is more sensitive to small infinitesimal changes. From this we can already deduce a certain initial regularity of the lower membrane. Then, the regularity of the upper membrane can be obtained by solving the obstacle problem in which the obstacle is given by the lower membrane. 
In order to obtain the optimal regularity we need to repeat these arguments several times.
A large part of the paper is devoted to obtaining estimates for various obstacle problems which are optimal with respect to the smoothness of the obstacle.
We first discuss the general case of operators that correspond to translation invariant kernels. Then we consider the special case of the fractional Laplacian. As mentioned above in the course of the paper we also treat the obstacle problem for translation invariant kernels  which is of independent interest.
 
 The paper is organized as follows. In Section 2 we formulate the two membranes problem and state precisely our results. In Section 3 we obtain the H\"older regularity of the minimizing pair. In Section 4 we develop the viscosity approach to the two membranes problem. In Section 5 we deal with the translation invariant kernels and finally in Section 6 we discuss the case of the fractional Laplacian. The Appendix is devoted to the proof of Schauder estimates for nonlocal equations.

\section{Main results}

\subsection{Notation}
Let $s \in (0,1)$ and let 
$k(x,y)$ be a symmetric, measurable kernel proportional to $|x-y|^{-n-2s}$, i.e.
$$0<\lambda \le k(x,y) |x-y|^{n+2s} \le \Lambda, \quad \quad k(x,y)=k(y,x).$$

Given a function $u \in L_{loc}^2$ we define its $H^s$ seminorm in $B_1$, the unit ball,  as 
$$\|u \|^2_{H^s(B_1)}:=\frac 1 2\int\int_{(\R^n\times \R^n) \setminus (\mathcal C B_1\times \mathcal C B_1)}\frac{(u(x)-u(y))^2}{|x-y|^{n+2s}} \, dx dy,$$ and if $\|u\|_{H^s(B_1)}<\infty$ we write $u \in H^s(B_1)$. Here for any set $E \subset \R^n$, we denote by $\mathcal C E$ its complement in $\R^n.$

It is not difficult to check that 
\begin{equation}\label{l2bound}\left \|u- \fint_{B_1} u \right \|_{L^2(\R^n,d \omega)} \le C  \|u \|_{H^s(B_1)}, \quad \quad  d \omega:=\frac{dx}{1+|x|^{n+2s}} .\end{equation}

Given two functions $u,v \in H^s(B_1)$
we define the ``inner product" of $u$ and $v$ with respect to the kernel $k$ as
\begin{equation}\label{E}
\mathcal E_{k}(u,v) : =\frac 1 2\int\int_{(\R^n\times \R^n) \setminus (\mathcal C B_1\times \mathcal C B_1)}(u(x)-u(y))(v(x)-v(y)) \, \, k(x,y) dx dy. 
\end{equation}

If $u$ minimizes the energy $\mathcal E_k(u,u)$ among all functions $u \in H^s(B_1)$ which are fixed outside $B_1$, say $u=u^0 \in H^s(B_1)$ outside $B_1$, then 
$$\mathcal E_k(u, \varphi) =0, \quad \quad \forall \varphi \in  H^s (B_1), \quad \mbox{with} \quad \varphi =0 \quad \mbox{outside} \quad B_1.$$

The last equality can be written in the sense of distributions as $\mathcal L_k u=0$ in $B_1$, with
$$<\mathcal L_k u, \varphi >:=- \mathcal E_k(u,\varphi) \quad  \forall \varphi \in C_0^\infty(B_1),$$
and formally $\mathcal L_k u$ can be written as the non-local operator
$$ \mathcal L_k u(x)=\int (u(y)-u(x))k(y,x)dy.$$

We wish to include the case when $k$ has order $s=1$. In this case the quadratic form $\mathcal E_k(u,u)$ is given by
\begin{equation}\label{EA}
\mathcal E_A(u,u)=\int_{B_1} (\nabla u)^T A(x) \nabla u \, dx,
\end{equation}
with $A(x)$ a symmetric $n \times n$ matrix satisfying $\lambda I \le A(x) \le \Lambda I$, and the linear operator associated to $\mathcal E _A$ is 
$$\mathcal L_A (u)= \textrm{div}(A(x) \nabla u) .$$

Finally, we notice the following scaling property of $\mathcal E_k$ after space dilation. Let $$\tilde u (x)=u (r x),$$ be the $1/r$ dilation of $u$ in the space variable. Then
$$\mathcal E_k(  u, v) = r^{n-2s} \, \, \mathcal E_{\tilde k}(  \tilde u, \tilde v)  $$
where in the double integral on the right we remove the contribution coming from $\mathcal C B_{1/r} \times \mathcal C B_{1/r}$ and the kernel $\tilde k(x,y):= r^{n+2s}k(rx,ry)$ is the rescaling on $k$, and therefore satisfies the same growth conditions as $k$.

\subsection{The two membranes problem -- General case, H\"older continuity  of the minimizers.}

We consider the two membranes obstacle problem in $B_1$ for operators corresponding to two different kernels $k_1$ and $k_2$ as above, with the order $s_1$ not necessarily equal to $s_2$. We look for a pair of functions $(u_1,u_2)$, with $u_2 \le u_1$ in $B_1$ and $u_1$, $u_2$ prescribed outside $B_1$, which minimizes the energy functional
\begin{equation}\label{F}\mathcal F(u_1,u_2):=\mathcal E_{k_1}(u_1,u_1) + \mathcal E_{k_2}(u_2,u_2) + \int_{B_1} u_1 f_1 + u_2 f_2 \, \, dx, \end{equation}
among all $ (u_1,u_2) \in \mathcal A.$

Here $f_i \in L^2(B_1)$ and $\mathcal A$ represents the set of admissible pairs,
$$\mathcal A=\left \{ (u_1,u_2)| \quad u_2 \le u_1, \quad u_i \in  H^{s_i}(B_1), \quad u_i=u^0_i \quad \mbox{outside} \quad B_1 \right\},$$
with $ u_i^0 \in H^{s_i}(B_1)$, $u^0_2 \le u^0_1$ in $ B_1$, a given pair of functions. 

With the convention in the Subsection above, we allow in the definition of the energy $\mathcal F$ also the cases when either one or both of the  $s_i$'s equal to 1, and we need to replace the quadratic form accordingly.

Since $\mathcal F$ is strictly convex, and $\mathcal F(u_1^0,u_2^0)< \infty$, we obtain the existence and uniqueness of a minimizing pair $(u_1,u_2)$ by the standard methods of the calculus of variations.
 
 \begin{prop}
 There exists a unique minimizing pair $(u_1,u_2) \in \mathcal A$ for the functional $\mathcal F$ in \eqref{F}. Moreover $u_i \in L^2(\R^n, d\omega_i)$ and $\sum_i \|u_i\|_{L^2(d\omega_i)} \leq C$ for a constant $C$ depending on the boundary data $u_i^0$ and on the $f_i$'s.
 \end{prop}
 
 We observe that to prove the $L^2$ bound for the minimizing pair, one uses \eqref{l2bound}.

 Notice that if $\varphi \ge 0$ and $\varphi \in C_0^\infty (B_1)$ then 
 $$(u_1 + \eps \varphi,u_2) \in \mathcal A \quad \mbox{ and}  \quad (u_1,u_2-\eps \varphi) \in \mathcal A,$$ 
 which gives 
 \begin{equation} \label{EL1}
 \mathcal L_{k_1} u_1 \le f_1, \quad \quad \mathcal L_{k_2} u_2 \ge f_2 \quad \mbox{in} \quad B_1,
 \end{equation}
in the sense of distributions, thus $\mathcal L_{k_1}u_1$, $\mathcal L_{k_2}u_2$ are Radon measures.

Moreover, if  $\varphi \in C_0^\infty (B_1)$ is not necessarily positive we still have $$(u_1+\eps \varphi,u_2+ \eps \varphi)\in \mathcal A,$$ hence
\begin{equation}\label{EL2}
\mathcal L_{k_1} u_1+ \mathcal L_{k_2} u_2= f_1+f_2   \quad \mbox{in} \quad B_1.
\end{equation}
Equations \eqref{EL1}-\eqref{EL2} together with the inequality $u_2 \le u_1$, can be viewed as the Euler-Lagrange characterization of the minimizing pair. 
 
In this paper we are concerned with the regularity of the minimizing pair $(u_1,u_2)$ and some properties of the free boundary $\Gamma$ which is defined as the boundary of the coincidence set, i.e.$$\Gamma:=\p \{u_1=u_2\} \cap B_1.$$

Our first result is the following interior H\"older regularity of the minimizing pair.

\begin{thm}\label{T1}
Assume $f_i \in L^{q_i}(B_1)$ with $q_i > \frac{n}{2s_i} $. Let $(u_1,u_2)$ be a minimizing pair. Then $u_i \in C^\alpha(B_1)$ and
$$\sum_i \|u_i\|_{C^\alpha(B_{1/2})} \le C  \sum_i \left(\|u_i\|_{L^2(d \omega_i)}+\|f_i\|_{L^{q_i}(B_1)}\right ),$$
with $\alpha$ and $C$ depending on $n$, $\lambda$, $\Lambda$, $s_i$, $q_i$.
\end{thm}

To obtain better regularity properties of the minimizing pair we need to require that the kernels $k_i$ are more regular, as in the next subsection.

\subsection{Translation invariant kernels -- Viscosity solutions and higher regularity}
We consider the case when $k$ is translation invariant, i.e.
$$k(x,y)=K(x-y), \quad \quad K(y)=K(-y),$$
and satisfies the natural growth condition of the gradient
\begin{equation*}
|\nabla K(y)| \le \frac{\Lambda}{|y|^{n+1+2s}}. 
\end{equation*}
 The integro-differential operator associated to this kernel can be written as
\begin{equation*}
\mathcal L_K w (x) := PV \int_{\R^n}(w(y)-w(x)) K(y-x) dy.
\end{equation*}
and the value $\mathcal L_K w(x)$ is well-defined as long as $w \in L^1(\R^n, d \omega)$ and $w$ is $C^{2s+\eps}$ at $x$.

In this case we show that the minimizing pair $(u_1,u_2)$ satisfies 
\begin{equation}\label{EL30} 
u_1 \ge u_2, \quad \quad \mathcal L_{K_1} u_1 \le f_1, \quad \quad \mathcal L_{K_2} u_2 \ge f_2 \quad \mbox{in} \quad B_1,
\end{equation}
\begin{equation}\label{EL40}
\mathcal L_{K_i}u_i=f_i \quad \mbox{on $\{u_2>u_1\}$,} \quad 
\sum_i \mathcal L_{K_i} u_i= \sum_i f_i   \quad \quad \mbox{in} \quad B_1,
\end{equation}
in the viscosity sense, and moreover these inequalities determine uniquely the pair $(u_1,u_2)$ (see Proposition \ref{evs}).

When the orders of the operators $\mathcal L_{K_i}$ are different we improve the result of Theorem \ref{T1} and obtain the $C^{1,\gamma}$ regularity of the pair $(u_1,u_2)$. Notice that the two membranes may interact, that is $\{u_1= u_2 \} \cap B_1 \ne \emptyset$ independently of the sign of $f_1$, $f_2$. We obtain the following result.

\begin{thm}\label{T2}
Assume $s_1<s_2$ and $u_i$ satisfy \eqref{EL30}-\eqref{EL40} with $f_i \in C^{0,1}(B_1)$. Then $u_i \in C^{\alpha_i}(B_1)$ with $\alpha _i >1$, 
$$\alpha_1=\max\{1,2s_1  \} + \eps_0, \quad \quad \alpha_2=\alpha_1 + 2(s_2-s_1)$$
and
$$\sum_i \|u_i\|_{C^{\alpha_i}(B_{1/2})} \le C  \sum_i \left(\|u_i\|_{L^1(d \omega_i)}+\|f_i\|_{C^{0,1}(B_1)}\right ),$$
with $\eps_0$ and $C$ depending on $n$, $\lambda$, $\Lambda$, $s_i$.
\end{thm}

\subsection{The obstacle problem for operators with translation invariant kernels.} In order to obtain Theorem \ref{T2} we study the obstacle problem for the operator $\mathcal L_K$ associated to a translation invariant kernel of order $2s$. We obtain the following result, of independent interest.
Assume that $u$, $\varphi$ are continuous in $B_1$, $u \in L^1(\R^n, d \omega)$, 
and
\begin{equation}\label{opti1intro}
u \ge \varphi \quad \mbox{in $B_1$} ,
\end{equation}
\begin{equation}\label{opti2intro}
\mathcal L_{K} u \le f \quad \mbox{in $B_1$}, \quad \mbox{and} \quad \quad \quad \mathcal L_{K} u = f \quad \mbox{in} \quad \{u>\varphi\} \cap B_1,
\end{equation} with $K$ of order $2s$ as at the beginning of subsection 2.3.

\begin{thm}\label{T20intro}
Let $u$ be a solution to \eqref{opti1intro}, \eqref{opti2intro}, and assume that
$$\|u\|_{L^1(\R^n, d \omega)}, \|\varphi\|_{C^\beta(B_1)}, \|f\|_{C^{0,1}(B_1)} \le 1.$$
for some $\beta \ne 2s$. 

Then $u \in C^\alpha(B_1)$ for $\alpha = \min \{\beta, \max \{1, 2s\}+ \eps_0 \}$ and
 $$\|u\|_ {C^\alpha(B_{1/2})} \le C,$$
where $\eps_0$ depends on $n$, $\lambda$, $\Lambda$, $s$, and the constant $C$ may depend also on $\beta$.
\end{thm}

\subsection{Fractional laplacian -- Optimal regularity and the geometry of the free boundary.} In the special case when $$K(y)=\frac{1}{|y|^{n+2s}}$$ the operator $\mathcal L_{K}$ 
reduces to the fractional Laplacian $\Delta^{s}$ and we obtain the optimal regularity of the solution.  As usual, we can 
characterize the points on the free boundary 
$$\Gamma:=\p \{u = \varphi\}  \cap B_1.$$
Precisely the set $\Sigma \subset \Gamma$ of {\it singular points} consists of those $y \in \Gamma$ such that $$(u - \varphi)(x)=o(|x-y|^{1+s}),$$ and $\Gamma \setminus \Sigma$ is the set of {\it regular points} (or stable points) of the free boundary.

\begin{thm}\label{T31intro}
Let $u$ be a solution to \eqref{opti1intro}, \eqref{opti2intro}, with
$$\|u\|_{L^1(\R^n, d \omega)}, \|\varphi\|_{C^\beta(B_1)}, \|f\|_{C^{\beta-2s}(B_1)} \le 1, \quad \quad \mbox{for some $\beta > 1 + s$.}$$
Then $u \in C^{1+s}(B_1)$ and
 $$\|u\|_ {C^{1+s}(B_{1/2})} \le C.$$
 Moreover, the free boundary $\Gamma$ is a $C^{1,\gamma}$ surface in a 
 neighborhood of each of its regular points. The constants $C, \gamma$ 
 depend on $n$, $s$, and $\beta$.
 
\end{thm}

Theorem \ref{T31intro} was obtained by Caffarelli, Salsa and Silvestre in \cite{CSS} . The main tool in the proof is to establish a version of 
Almgren's frequency formula for the ``extension" of $u$ to $\R^{n+1}$. However, Theorem \ref{T31intro} is proved in \cite{CSS} in the case when $\varphi \in C^{2,1}$ ( i.e. $\beta=3$). When $s=1/2,$ Guillen proved Theorem \ref{T31intro} in \cite{G}. In Section 6 we
show that the Almgren's monotonicity formula still holds when $\beta>1+s$ and therefore sharpen the result in \cite{CSS} and obtain Theorem \ref{T31intro}.

Theorem \ref{T31intro} yields the following result for the two-membrane problem. When $$K_1(y)=\frac{1}{|y|^{n+2s_1}}$$ we obtain the optimal regularity of the 
minimizing pair, i.e. $u_1 \in C^{1,s_1} $ and $u_2 \in C^{1+2s_2-s_1}$ and we can 
characterize the points on the free boundary 
$$\Gamma:=\p \{u_1 =u_2\}  \cap B_1,$$
as in the obstacle problem.

\begin{thm} [Optimal regularity] \label{T3}
Assume that the hypotheses of Theorem $\ref{T2}$ hold and $K_1$ is as above. Then the conclusion  of Theorem $\ref{T2}$ holds with $\alpha_1=1+s_1$.
Moreover, the set of regular points of the free boundary $\Gamma$ is locally a $C^{1,\gamma}$ surface.
\end{thm} 

\section{The proof of Theorem \ref{T1}}

In this section we prove the H\"older regularity of the minimizing pair $(u_1, u_2).$ The parameters $\lambda, \Lambda, n, s_1,s_2$ are called universal and any constant depending only on these parameters is called universal as well and it is usually denoted by $C, c$ (though it may change from line to line).

\medskip

\textit{Proof of Theorem $\ref{T1}$}. The proof follows from the standard De Giorgi iteration technique. For simplicity we sketch it for $f_i=0$ and $s_i<1$, since the arguments carry on without difficulty to the case of nonzero $f_i$'s and when one or both operators are local.

Assume that $$s_2 \ge s_1.$$

{\it Step 1. Caccioppoli inequality.} 
%Since $u_2$ is a subsolution for the $\mathcal L_{k_2}$ operator, we find that the ``upper" truncation $u_2^+$ satisfies these inequalities.
%Below we show that the ``lower" truncation $u_2^-$ satisfies a Caccioppoli inequality as well.
Let $\varphi$ be a cutoff function supported in $B_1$. The key observation is that for $\eps <1$,
$$\left(u_1+\eps \varphi^2 u_1^-,u_2+\eps \varphi^2 u_2^-\right)  \in \mathcal A.$$
Using the minimality of the pair $(u_1,u_2)$, we let $\eps \to 0$ and obtain
\begin{equation}\label{ein}
\mathcal E_{k_1}(u_1,\varphi^2 u_1^-) +  \mathcal E_{k_2}(u_2,\varphi^2 u_2^-) \ge 0.
\end{equation}

Notice that
$$- \mathcal E_{k}(u,\varphi^2 u^-)= \mathcal E_{k}(u^-,\varphi^2 u^-) + F_k(u),$$
and $$F_k(u):=- \mathcal E_{k}(u^+,\varphi^2 u^-)=2\int\int \varphi^2(x)\, u^+(x)\, u^-(y) \, \, k(x,y) dx dy \ge 0.$$

We use the identity
$$(a-b)(p^2a-q^2b)=(ap-bq)^2-ab(p-q)^2,$$
thus
$$\mathcal E_{k}(u^-,\varphi^2 u^-)  = 
 \mathcal E_{k}(\varphi u^-,\varphi u^-) - I_k(u)$$
with
$$I_k(u)=\int \int u^-(x)u^-(y)(\varphi(x)-\varphi(y))^2k(x,y) dx dy  \ge 0.$$ 
The identities above give $$ \mathcal E_{k}(\varphi u^-,\varphi u^-) + F_k(u) = -\mathcal E_{k}(u,\varphi^2 u^-) + I_k(u).$$

Next we bound above $I_k(u)$.

Assume that $\varphi$ is the usual cutoff function with $\varphi=1$ in $B_r$ and $\varphi=0$ outside $B_{r+\delta/2}$ for some $r \in (0,1-\delta]$.
When both $x$ and $y$ are in $B_{r+\delta}$ we use that
$$u^-(x)u^-(y)(\varphi(x) -\varphi(y))^2 \leq C \delta^{-2} [(u^-(x))^2+(u^-(y))^2 ]|x-y|^2.$$
When $x \in B_{r+\delta/2}$ and $y$ lies outside $B_{r+\delta}$ (and symmetrically the other case),
we use that 
$$k(x,y) \leq C\delta^{-n-2s} \omega(y).$$
Thus, we see that $I_k(u)$ is bounded above by
$$I_k(u) \le C \delta^{-2} \int_{B_{r+\delta}} (u^-)^2 dx + C \|u^-\|_{L^2(d \omega)} \delta^{-n-2s} \int_{B_{r+\delta}} u^- dx.$$
In this last inequality we used that $ \|u^-\|_{L^1(d \omega)} \leq  C \|u^-\|_{L^2(d \omega)}.$
We use these relations for $u_1$ and $u_2$ in the energy inequality \eqref{ein} together with the fact that $u_2^- \ge u_1^-$ in $B_1$. We obtain the desired Caccioppoli inequality for $u_2^-$:

\begin{equation}\label{Cin}
 \mathcal E_{k_2}(\varphi u^-_2,\varphi u_2^-) + F_{k_2}(u_2)\le  C_0 \delta^{-n-2} \int_{B_{r+\delta}} [(u^-_2)^2 + M_0 u_2^- ]\, \, dx.
\end{equation}
with
$$ M_0:= \|u_1^-\|_{L^2(d \omega_1)}+\|u_2^-\|_{L^2(d \omega_2)},$$
and $C_0$ universal.
More generally if $v_m = u_2+m$, we have
\begin{equation}\label{Cinm}
 \mathcal E_{k_2}(\varphi v_m^-,\varphi v_m^-) + F_{k_2}(v_m)\le  C_0 \delta^{-n-2} \int_{B_{r+\delta}} [(v_m^-)^2 + A_m v_m^-] \, \, dx,
\end{equation}
and $$M_m = \|(u_1+m)^-\|_{L^2(d \omega_1)}+\|(u_2+m)^-\|_{L^2(d \omega_2)}.$$
Moreover, for all constants $m \geq 0$, $M_m \leq M_0$ hence 
\begin{equation}\label{Cinm0}
 \mathcal E_{k_2}(\varphi v_m^-,\varphi v_m^-) + F_{k_2}(v_m)\le  C_0 \delta^{-n-2} \int_{B_{r+\delta}} [(v_m^-)^2 + M_0 v_m^-] \, \, dx.
\end{equation}

\begin{rem}\label{rem} Since $u_2$ is a subsolution for the $\mathcal L_{k_2}$ operator, $v_m^+:= (u_2-m)^+$ satisfies the same inequality \eqref{Cinm} with the constant $M_m$ replaced by $\|(u_2-m)^+\|_{L^2(d \omega_2)}$.\end{rem}

\medskip

{\it Step 2. The first De Giorgi lemma}. We write the first De Giorgi type lemma and provide a sketch of the proof (see also Lemma 3.1 in \cite{CCV}.)

\begin{lem}[$L^\infty$ bound]\label{l11} Assume $v_m : = u_2 +m$ satisfies \eqref{Cinm0} for all $0 \leq m \leq 1$ and some $M_0>0.$ There exists $\eps_0$ depending on the universal parameters and $M_0$ such that
 if $$\|u_2^-\|_{L^2(B_1)} \le \eps_0 (M_0),$$then $$ u_2^- \le 1 \quad \mbox{in $B_{1/2}$}.$$
\end{lem}
\begin{proof} We apply \eqref{Cinm0},  with ($j \geq 2$)
$$m=m_j:= 1- 2^{-j}, \quad r=r_j:= \frac 1 2 +2^{-j}, \quad \delta=\delta_j:=2^{-j}.$$ Using that $F_{k_2}(v_{m_j}) \geq 0$ together with Sobolev inequality we get ($1/2^*= 1/2-s_2/n$)
\begin{align}\label{Cinmk}
\left(\int_{B_{r_j}}(v_{m_j}^-)^{2^*} \right)^{2/2^*} & \le  C_0 \delta_j^{-n-2} \int_{B_{r_j+\delta_j}} [(v_{m_j}^-)^2 + M_0 v_{m_j}^-] \, \, dx\\ \nonumber & := R_j.
\end{align}

Call,
$$a_j:= \int_{B_{r_j}}(v_{m_{j}}^-)^2$$
and 
$$A_j:= \{v_{m_{j}}<0\} \cap B_{r_j}.$$
Applying Holder's inequality to the left-hand-side of \eqref{Cinmk} and using the notation above we get
\begin{equation}\label{ak}
a_{j} \leq |A_{j}|^{\frac{2 s_2}{n}}R_j \leq  |A_{j}|^{\frac{2 s_2}{n}} (C_0 2^{Mj} a_{j-1} + M_0 a_{j-1}^{1/2} |A_{j}|^{1/2}),
\end{equation} for some large $M.$ Since on $A_{j}$, $v_{m_{j-1}} < -2^{j}$, we easily obtain that
$$a_{j-1} \geq |A_{j}|2^{-2j}.$$ Thus, \eqref{ak} gives (for some positive $\sigma$ and with $\bar C$ depending on the universal constants and $M_0$)
$$a_{j} \leq \bar C 2^{Mj} a_{j-1}^{1+\sigma}.$$ Standard De Giorgi iteration gives that if $a_2$ is small enough (depending on $\bar C$) $a_j \to 0$ as $j \to \infty$ and from this we deduce our claim.
\end{proof}

Our minimization problem remains invariant after multiplication with a constant. Thus, after multiplication with a small constant we may apply Lemma \ref{l11} and obtain the $L^\infty$ bound for $u_2$ in $B_{1/2}.$

\medskip

{\it Step 2. The second De Giorgi lemma and the H\"older continuity of $u_2$}. In order to obtain the H\"older continuity of $u_2,$ we need to iterate the next Lemma \ref{l12}, and  this is point where we need $s_2 \ge s_1$. 

Notice that in general the minimization problem is not invariant after a dilation in the space variable. Indeed, if $\tilde u_i(x)=u_i(\rho x)$ then
$$(\tilde u_1,\tilde u_2) \quad \mbox{minimizes the energy} \quad \mathcal  \rho^{2(s_2-s_1)} \mathcal E_{k_1}(\tilde u_1,\tilde u_1) + \mathcal E_{k_2}(\tilde u_2,\tilde u_2) .$$ 
Thus if $\rho \le 1$ the arguments above apply and the Caccioppoli inequality \eqref{Cin} holds for $\tilde u_2$ with 
$$\tilde M = \rho^{2(s_2-s_1)} \|\tilde u_1^-\|_{L^2(d \omega_1)}+\|\tilde u_2^-\|_{L^2(d \omega_2)} \le  \tilde M_0:= \|\tilde u_1^-\|_{L^2(d \omega_1)}+\| \tilde u_2^-\|_{L^2(d \omega_2)}.$$

Notice also that \begin{equation}\label{L2}\|\tilde u_i\|_{L^2(\R^n \setminus B_{1/\rho}, d \omega_i)} \sim  \rho^{s_i} \|u_i\| _{L^2(\R^n \setminus B_ 1,d \omega_i)}.\end{equation}

\begin{lem}[Oscillation decay]\label{l12}
Assume that $u_2$ satisfies  \eqref{Cinm}, for all constants $m$. Suppose that for some $R\geq 1$, 
$$u_1 \geq u_2 \quad \text{in $B_R$,}$$
and
$$|u_2| \le 1 \quad \mbox{in $B_R$}, \quad \quad \|u_1^-\|_{L^2(\R^n \setminus B_R,d \omega_1)} \le \mu, \quad \|u_2^-\|_{L^2(\R^n \setminus B_R,d \omega_2)}  \le \mu, $$ with $\mu$ universal.
Then in $B_1$ either $u_2 \le 1- \mu$ or $u_2 \ge -1 + \mu$.
\end{lem}
\begin{proof} Let us assume first that
$$|u_2| \leq 1 \quad \text{in $\R^n$.}$$ 
Assume that 
\begin{equation}\label{measure}
|\{u_2 >0\} \cap B_1| \geq \frac 1 2 |B_1|.
\end{equation}
We will show that there is a universal constant $\eta$ such that $u_2 \geq -1+\eta$ in $B_1.$ 
Let,
$$v_j:= 2^j (u_2 + (1-2^{-j})), \quad A_j:= \{v_j <0\} \cap B_1.$$ We aim to show that there is a large enough $j$ such that 
\begin{equation}\label{goal}
|A_{j+1}| \leq \delta_0
\end{equation}
with $\delta_0$ universal to be made precise later.

Assume by contradiction that 
$$|A_{j+1}| > \delta_0$$
and let us choose $\delta << \delta_0$ so that
\begin{equation}\label{m2}
|A_{j+1} \cap B_{1-\delta}| \geq \frac{\delta_0}{2}.
\end{equation}
By Caccioppoli inequality \eqref{Cin} for $v_j$ we obtain
\begin{equation}\label{FC}
F_{k_2}(v_j) \leq C \delta^{-n-2}
\end{equation}
where we have used that $v_j^- \leq 1$ in $\R^n,$ and that $u_1 \geq u_2,$ so that the corresponding constant $M_j$ in \eqref{Cin} is bounded by a universal constant $\bar M$.

On the other hand,
\begin{align*}
F_{k_2}(v_j):= & 2\int\int \varphi^2(x)\, v_j^+(x)\, v_j^-(y) \, \, k_2(x,y) dx dy \ge \\
& c \int_{B_1}v_j^+(y)dy \int_{A_{j+1} \cap B_{1-\delta}} v_j^-(x)dx \ge \\
&c( 2^{j}-1)|A_{j+1} \cap B_{1-\delta}||B_1| \geq  2^j c \delta_0. 
\end{align*}
In the third inequality above we used that
$$v_j^- \geq \frac 1 2 \quad \text{on $A_{j+1}$}$$
and \eqref{measure}.

Thus, \eqref{m2} is violated if $j$ is large enough. Denote such $j$ by $\bar j.$

Now we can apply Lemma \ref{l11} to $v_{\bar j+1}$ and choose $\delta_0=\eps_0(2\bar M)$ where $\bar M$ is the universal constant that bounds all the $M_j$'s (as observed above.) We obtain the conclusion with $\eta=2^{-(\bar j+1)}.$

Now assume that $|u_2| \leq 1$ in $B_R$ and $u_1 \geq u_2$ in $B_R$, for $R \geq 1.$ Let also
$$\|u_i^-\|_{L^2(\R^n \setminus B_R, d\omega_i)} \leq \eps.$$ Then, for $\eps$ small enough the argument above still holds for the fixed $\bar j$. Indeed  one can still guarantee that $M_{\bar j} \leq 2\bar M$ for $\eps$ small enough. 

Finally, if \eqref{measure} does not hold, then we can work with the Caccioppoli inequality for $(u_2-m)^+$ and obtain that $u_2$ separates from the top (see Remark \ref{rem}).
\end{proof}

\smallskip

Finally we can iterate Lemma \ref{l12} and obtain the interior $C^\alpha$ Holder continuity of $u_2$. Indeed, after a multiplication by a constant we may assume that $\|u_i\|_{L^2(d \omega_i)}$ are sufficiently small and $|u_2| \le 1$ in $B_{1/2}$. Then we perform an initial dilation of size $R_0$, and we may apply Lemma \ref{l12}. Notice that the hypotheses are satisfied thanks to \eqref{L2}. Moreover it is easy to check that the hypotheses hold for the sequence of H\"older rescalings $$\left(\frac{2}{2-\mu} \right)^{m-1} u_2(R_0^{-m} x) + const, \quad \quad m=1,2,..$$ provided that $R_0$ is chosen sufficiently large, and we may apply Lemma \ref{l12} indefinitely.

\medskip

{\it Step 3. The second De Giorgi lemma and the Holder continuity of $u_1$}. Next we obtain the H\"older continuity of $u_1$ by thinking that $u_2 \in C^\alpha$ is a fixed obstacle lying above, and $u_1$ minimizes $\mathcal E_{k_1}(u_1,u_1)$ among admissible functions. 

Notice that since $|u_2| \leq 1$ and $u_1 \geq u_2$ we can obtain an $L^\infty$ bound for $u_1$ by applying the (standard) first De Giorgi lemma to $(u_1-1)^+$. Indeed in the set $u_1>1$, $u_1$ solves the equation $\mathcal L_{k_1} u_1=0.$

The H\"older continuity of $u_1$ follows by iterating the following version of the oscillation decay lemma.

\begin{lem}\label{l13} Assume that for some $R \ge 1$
$$|u_1| \le 1 \quad \mbox{in $B_R$}, \quad \quad \|u_1\|_{L^2(\R^n \setminus B_R,d \omega)} \le \mu, \quad osc_{B_1} u_2 \le 1/4.$$
Then in $B_1$ either $u_1 \le 1- \mu$ or $u_1 \ge -1 + \mu$.
\end{lem}

The proof of Lemma \ref{l13} is a variation of the proof above.
Indeed, if $u_2(0) \ge -\frac 12$ then the conclusion is obvious since $u_1 \ge u_2 \ge -\frac 34.$ 

If $u_2(0) \le -\frac 1 2$, we distinguish two cases. When $|\{u_1>0\} \cap B_1| > 1/2$, we use that $\mathcal L_{k_1}u_1 \le 0$ hence we apply De Giorgi technique to conclude that $u_1\geq -1+\mu.$

Otherwise, since $u_2 \le -\frac 14$ in $B_1$, $u_1$ is not constrained in the set $\{u_1>0\}$ and $\mathcal L_{k_1} u_1=0$ there. Again,  we can apply De Giorgi technique and conclude $u_1 \leq 1-\mu.$
\qed

\section{Translation invariant kernels and viscosity solutions}

In this section we investigate further properties of the minimizing pair $(u_1,u_2)$ when the kernels $k_i$ are more regular. 
More precisely, from now on we assume that the kernel $k$ used in the definition of the energy $\mathcal E_k$ in \eqref{E} is translation invariant i.e.
$$k(x,y)=K(x-y).$$
Here the kernel $K$ satisfies $K(y)=K(-y)$ and it is comparable to the kernel of $(-\Delta)^s$ i.e.
\begin{equation}\label{Kc}
 \frac{\lambda}{|y|^{n+2s}} \leq K(y) \leq  \frac{\Lambda}{|y|^{n+2s}}, \quad 0<\lambda \leq \Lambda. 
\end{equation}
 The integro-differential operator associated to this kernel can be written as
\begin{equation}\label{I}
\mathcal L_K w(x) := PV \int_{\R^n}(w(y)-w(x)) K(y-x) dy.
\end{equation}
Notice that the value $\mathcal L_K w(x)$ is well-defined as long as $w \in L^1(\R^n, d \omega)$ and $w$ is $C^{1,1}$ at $x$.

In the case $s=1$, of local operators defined in \eqref{EA}, we assume that the matrix $A$ is constant, and therefore $\mathcal L_A$ is a second order operator with constant coefficients.

\subsection{Viscosity properties of the minimizing pair}

To study further regularity of the minimizing pair, we adopt the point of view of viscosity solutions.

\begin{defn} Given a function $w: \R^n \to \R$, upper (lower) semicontinuous in $\bar B_1$ and a $C^2$ function $\phi$ defined in a neighborhood $N$ of a point $x \in B_1$, we say that $\phi$ touches $w$ by above (resp. below) at $x$ if 
$$\phi(x)=w(x), \quad \phi(y) > w(y)\quad  (\phi(y) < w(y)) \quad \text{for every $y \in N \setminus \{x\}$}.$$
\end{defn}

We remark that at any point $x$ where $w$ is touched by above or below, $\mathcal L_K w (x)$ is well-defined, though it may be infinite. Indeed, say $w$ is touched by below by $\phi$ at $x$ then 
$$\mathcal L_K w(x)=\int_0^\infty a_w(r) r^{-1-2s}dr \in (-\infty, +\infty]$$
where $a_w(r)$ represents the averages of $w$ on $\p B_r$
$$a_w(r)= \fint_{\p B_r(x)} (w(y)-w(x))K(y-x)r^{n+2s}dy$$
and for $r$ small (since $K$ is symmetric) $$a_w(r) \geq a_\phi(r) \geq -Cr^2.$$
\begin{defn} A function $w: \R^n \to \R$, upper (lower) semicontinuous in $\bar B_1$, is said to be a viscosity subsolution (supersolution) to $\mathcal L_K w=f,$ $f$ continuous in $B_1$, and we write $\mathcal L_K w \geq f$ ($\mathcal L_K w \leq f$), if at any point $x \in B_1$ where $w$ is touched by above (resp. below) by a quadratic polynomial $P$, we have
$$\mathcal L_K w (x) \geq f(x), \quad (\mathcal L_K w (x) \leq f(x)).$$
\end{defn}

A viscosity solution is a function $w$ that is both a subsolution and a supersolution.

Next we show that distributional supersolutions (subsolutions) are also viscosity supersolutions (subsolutions). We sketch the proof since we will use the same argument in a slightly different context.

\begin{lem}\label{visc}
Assume that $\mathcal L_K w \le f$ in the distribution sense with $w$, $f$ continuous functions in $B_1$. Then $\mathcal L_K w \le f$ in the viscosity sense. 
\end{lem}

\begin{proof} Assume for simplicity that $f=0$. Let $P$ be a quadratic polynomial touching $w$ strictly by below at say $0$. 
Let $P_\eps:= P+\eps$
 and denote by
 $$w_\eps:= \max\{w, P_\eps\}.$$
 and,
 $$ \varphi_\eps:= w_\eps - w \ge 0.$$

From the hypothesis $\mathcal E_K(\varphi_\eps, w) \geq 0$ thus
 $$\mathcal E_K(\varphi_\eps,w_\eps)=\mathcal E_K(\varphi_\eps, w) + \mathcal E_K(\varphi_\eps,\varphi_\eps)  \geq 0.$$

Since on the support of $\varphi_\eps$ we have that $w_\eps$ is $C^{1,1}$ by below, we can integrate by parts 
$\mathcal E_K(\varphi_\eps,w_\eps)$ and obtain
 \begin{equation}\label{small}\int_{A_\eps} \varphi_\eps(x) \mathcal L_K w_\eps(x)dx \leq 0,\end{equation}
 where $A_\eps := \{x : w < P_\eps\}$. Fix $\delta>0$, thus $A_\eps \subset B_\delta,$ for all $\eps$ small. We use that $w_\eps \ge P_\eps$ in $B_\delta$, $w_\eps=w$ outside $B_\delta$,  hence for $x \in A_\eps$, 
 \begin{align}\label{later}\mathcal L_K w_\eps (x) 
%&= \int_{A_\eps}(P_\eps(y)-P_\eps(x))K(y-x)dy + \int_{\R^n \setminus A_\eps}(w(y)-P_\eps(x))K(y-x)dy \\
 &\geq \int_{B_\delta}(P_\eps(y)-P_\eps(x))K(y-x)dy + \int_{\R^n \setminus B_{\delta}}(w(y)-P_\eps(x))K(y-x)dy\\
 \nonumber & \geq \int_{\R^n \setminus B_{\delta}}(w(y)- w(0))K(y)dy + o_\eps(1) + O(\delta^{2-2s}), \quad \text{as $\eps \to 0$,} 
 \end{align} with $o_\eps(1) \to 0$ as $\eps \to 0.$
Combining this estimate with \eqref{small}, and using that $\varphi_\eps \geq 0,$ we obtain that 
$$\mathcal L_K w(0) \leq 0,$$
after letting $\eps$ and then $ \delta$ go to zero.
 \end{proof}

By Lemma \ref{visc}, if $(u_1,u_2)$ is a minimizing pair and $f_i$ are continuous functions then (see \eqref{EL1}-\eqref{EL2})
\begin{equation}\label{e1}
\mathcal L_{K_1} u_1 \le f_1, \quad  \mathcal L_{K_2} u_2 \ge f_2, \quad \mbox{in $B_1$},
\end{equation}
$$\mathcal L_{K_1} u_1 = f_1, \quad  \mathcal L_{K_2} u_2 = f_2, \quad \mbox{in the open set $\{u_1>u_2\}$}.$$
in the viscosity sense. Next we prove a similar statement in the closed set 
\begin{equation}\label{Es}
E:=\{u_1=u_2\}.
\end{equation}

\begin{lem}\label{vi}
Assume that $u_2$ is touched by below at a point $x_0 \in \{u_1=u_2\} \cap B_1$ by a $C^2$ function. Then
$$\mathcal L_{K_1}u_1(x_0) + \mathcal L_{K_2}u_2(x_0) \leq f_1(x_0)+f_2(x_0).$$ 
\end{lem}

%The statement can be interpreted as stating that \mathcal L_{K_1}u_1(x_0)
We remark that, since $u_1 \ge u_2$, $u_1$ is touched by below at $x_0$ by the same $C^2$ function, thus $\mathcal L_{K_1}u_1(x_0)$ is well defined.

\begin{proof} We argue as above. Assume for simplicity that $f_1=0$, $f_2=0$, $x_0=0$, and let $P$ be a quadratic polynomial touching $u_2$ strictly by below at $0$. Let $P_\eps:= P+\eps$
 and denote by
 $$u_i^\eps:= \max\{u_i, P_\eps\}, \quad \varphi_i^\eps:= u_i^\eps - u_i.$$
 
 By minimality,
 $$\sum_i \left( \mathcal E_{K_i}(u_i^\eps, u_i^\eps) - \mathcal E_{K_i}(u_i, u_i) \right) \ge 0,$$
thus
$$\sum_i \mathcal E_{K_i}(\varphi_i^\eps, u_i^\eps) \ge \frac 12 \sum_i   \mathcal E_{K_i}(\varphi^\eps_i, \varphi^\eps_i) \ge 0$$
 
 After integrating by parts the terms $\mathcal E_{K}(\varphi^\eps, u^\eps)$ we get,
$$\sum_i  \int_{A^\eps_i} \varphi_i^\eps \, \,( \mathcal L_{K_i} u_i^\eps) \, dx \leq 0,$$
 where $A^\eps_i:=\{u_i < P_\eps\}$.
 Arguing as \eqref{later} in Lemma \ref{visc} we obtain that
 $$\sum_i \left(\int_{A_i^\eps} \varphi_i^\eps \right ) \left( \mathcal L_{K_i} u_i(0) + o_\delta(1)+o_\eps(1)\right) \leq 0,$$
 with $$o_\delta(1) \to 0 \quad \text{as $\delta \to 0$,} \quad o_\eps(1) \to 0 \quad \text{as $\eps \to 0$}.$$
 Since we already know that $\mathcal L_{K_1}  u_1(0) \leq 0$ and also $0<\varphi_1^\eps \leq \varphi_2^\eps$, we get the desired inequality after dividing by $\int \varphi_2^\eps$ and then letting $\eps \to 0$, $\delta \to 0$.
 \end{proof}

\subsection{Viscosity formulation of the two membranes problem} 

Next we show that we can formulate the two membranes problem in a non-variational setting. With this approach we may consider the two membranes problem for nonlinear operators $Iu$ or $F(D^2u)$ (instead of $\mathcal L_K$) which do not have necessarily a variational structure. 

Below we show that the following conditions in $B_1$ 
\begin{equation}\label{EL3} 
u_1 \ge u_2, \quad \quad \mathcal L_{K_1} u_1 \le f_1, \quad \quad \mathcal L_{K_2} u_2 \ge f_2 
%\quad \mbox{in} \quad B_1,
\end{equation}
\begin{equation}\label{EL4}
\mathcal L_{K_i}u_i=f_i \quad \mbox{in $\{u_2<u_1\}$}, \quad \quad \mathcal L_{K_1} u_1+ \mathcal L_{K_2} u_2= f_1+f_2,  
% \quad \mbox{in} \quad B_1
\end{equation}
determine the pair $(u_1,u_2)$ uniquely. 

We always assume that outside $B_1$, $u_i=u_i^0$ are prescribed with 
$u_i^0 \in L^1(\R^n,d\omega)$ and continuous near $\p B_1$, and $u^0_1\ge u^0_2$ near $\p B_1$. Also we assume that $f_i$'s are continuous and bounded in $B_1$. 

\begin{defn}\label{dvis}
We say that $(w_1,w_2)$ is a viscosity subsolution to \eqref{EL3}-\eqref{EL4} if $w_i$  are continuous in a neighborhood of $ \overline{ B}_1$,  and in $B_1$ we have $w_2 \le w_1$, and  
 \begin{equation}\label{EL5} 
 \mathcal L_{K_2} w_2 \ge f_2,
\end{equation}
\begin{equation}\label{EL6}
\mathcal L_{K_1} w_1+ \chi_E \mathcal L_{K_2} w_2 \ge f_1+ \chi_E f_2   \quad \mbox{with} \quad E:=\{w_1=w_2\}.
\end{equation}
\end{defn}

\

Similarly, we define the notion of viscosity supersolution for the two membranes problem. Equation \eqref{EL6} is understood as a differential inequality for $w_1$ which depends on $w_2$. Notice that at a point $x_0 \in E$ where $w_1$ has a tangent $C^2$ function $\phi$ by above, the same function is tangent also to $w_2$ at $x_0$, and therefore \eqref{EL6} provides an integro-differential inequality involving $\phi$ at $x_0$. Precisely we require that when we replace $w_i$ by $\phi$ in any $\delta$ neighborhood of $x_0$ the inequality \eqref{EL6} is satisfied at $x_0$.
 
In the next lemma we show that even though the inequality \eqref{EL6} contains the discontinuous term $\chi_E$,  the notion of subsolution is preserved under uniform limits.
 
 \begin{lem}Assume that $(w_1^k,w_2^k)$ is a sequence of subsolutions with right hand sides 
 $(f_1^k,f_2^k)$. Assume that $w_i^k$, $f_i^k$ converge uniformly on compact sets of $B_1$ to 
 $\bar w_i$, $\bar f_i$ and that $w_i^k \to \bar w_i$ weakly in $L^1(\R^n,d \omega)$. 
 Then $(\bar w_1,\bar w_2)$ is a subsolution.
 \end{lem} 

\begin{proof}
Clearly $\bar w_2$ satisfies \eqref{EL5}.
Assume that $\phi \in C^2$ touches strictly by above $\bar w _1$ at some point $\bar x$. Denote $\bar E := \{\bar w_1= \bar w_2\}.$

If $\bar x \notin \bar E$ then we obtain as usual
\begin{equation} \label{30}
\mathcal L_{K_1} \bar w_1 (\bar x) \ge \bar f_1(\bar x),
\end{equation}
and we are done.

If $\bar x \in \bar E$ we need to show that 
$$ \sum_i \mathcal L_{K_i} \bar w_i (\bar x) \ge \sum_i f_i(\bar x).$$
We slide the graph of $\phi$ by above till it touches $w_1^k$ at $x_k$, and then $x_k \to \bar x$. We distinguish two cases: either $x_k \in E_k$ or $x_k \notin E_k$ for infinitely many $k$'s. 
In the first case we obtain the inequality above by writing it for the $w_i^k$ at $x_k$ and letting $k \to \infty$. In the second case we obtain \eqref{30} which combined with \eqref{EL5} for $\bar w_2$ gives the desired inequality again. 
\end{proof}
\begin{lem}
Assume that $(w_1^k,w_2^k)$, $k=1,2$ are two pairs of subsolutions,
and let $\bar w_i=\max_k w_i^k$, $\bar f_i=\min_k f^k_i$. 
Then $(\bar w_1, \bar w_2)$ is a subsolution.
\end{lem}

\begin{proof}
Notice that $\bar E:=\{\bar w_1=\bar w_2\} \subset E_1 \cup E_2,$ $E_k:= \{w_1^k=w_2^k\}, k=1,2,$ and then the rest of the proof it is straightforward to check.
\end{proof}

In view of the lemma above we can use the standard method of sup-convolutions (see \cite{CC}, \cite{CS1}) 
and approximate a subsolution $(w_1,w_2)$ with right hand side $(f_1,f_2)$ by a sequence of 
semiconvex subsolutions $(w_1^\eps,w_2^\eps)$ and right hand side $(f_1^\eps,f_2^\eps)$. 

Precisely, $(w_1^\eps, w_2^\eps)$ satisfies:

a) has the same boundary data outside $B_1$ as the original pair, 

b) is a subsolution in $B_{1-\eps}$ and each $w_i^\eps$ is uniformly $C^{1,1}$ by 
below. 

c) $w_i^\eps \to w_i$, $f_i^\eps \to f_i$ uniformly in $\bar B_1$ as $\eps \to 0$.

\

Next we prove the following comparison principle.

\begin{lem}[Maximum principle]
Assume that $(w_1,w_2)$ is a subsolution and $(v_1,v_2)$ is a supersolution to 
\eqref{EL3}-\eqref{EL4} and $w_i \le v_i$ outside $B_1$. Then $w_i \le v_i$ also in $B_1$.
\end{lem}

\begin{proof}
We translate down the graphs of the pair $(w_1,w_2)$ in $\bar B_1$ and then we move them up till either $w_1$ touches $v_1$ or $w_2$ touches $v_2$ for the first time.

Assume by contradiction that the first contact point occurs in the interior of $B_1$. After regularizing the functions $w_i$, $v_i$ as above and relabeling the translates by $w_1$, $w_2$ we may assume we are in the following situation: 
$$w_i \le v_i, \quad w_2(x_0)=v_2(x_0) \quad \mbox{for some $x_0 \in B_1$,} $$
$(w_1,w_2)$ is a strict subsolution and $(v_1,v_2)$ is a strict supersolution at $x_0$, and $w_i$, $v_i$ are $C^{1,1}$ at $x_0$. If  at least one of the operators is local then we may assume that all the functions are $C^2$ at $x_0$ after subtracting locally a small linear function from one of the pairs, see \cite{CC}. Let $E_w:=\{w_1=w_2\}$, $E_v:=\{v_1=v_2\}$ and we distinguish 2 cases.

\smallskip

\noindent
{\it Case 1:}  $x_0 \notin E_v$. Then we contradict the inequalities for $ \mathcal L_{K_2} w_2$ and $\mathcal L_{K_2} v_2$ at $x_0$. 

\noindent
{\it Case 2:} $x_0 \in E_v$. Then $$w_1(x_0)\le v_1(x_0) =v_2(x_0)=w_2(x_0),$$ 
thus $x_0 \in E_w$ as well. Now we contradict the inequalities for the sum of the two operators at $x_0$.
\end{proof}

\begin{prop}[Existence and uniqueness of viscosity solutions]\label{evs}
Let $u_i^0 \in L^1(\R^n, d \omega_i)$ be continuous in a neighborhood of $\p B_1$, 
and let $f_i$ be continuous and bounded in $B_1$. 
Then there exists a unique viscosity solution pair $(u_1,u_2)$ to the 
two membranes problem \eqref{EL3}-\eqref{EL4}.
\end{prop}

\begin{proof}
The proof follows the standard Perron's method and we will not sketch the details. 
We only mention that the continuity of $u_i^0$ in a neighborhood of $\p B_1$ 
allows us to construct continuous upper and lower barriers for the subsolutions and 
supersolutions (see \cite{RS}). Using this we can replace each subsolution by a larger subsolution with a fixed 
modulus of continuity in $\bar B_1$, and therefore the largest subsolution will have the same 
modulus of continuity.
\end{proof}

\subsection{ The case of different order operators.}

Next we establish the $C^{2s_2-\eps}$ interior regularity of $u_2$ in the case when $s_2>s_1$.

\

Let $(u_1,u_2)$ be a viscosity solution in $B_2$, and assume that $$\|u_i\|_{L^1(d \omega_i)} \le 1, \quad \quad \|f_i\|_{L^\infty(B_2)} \le 1.$$

Since $u_2$ is a subsolution, we use the weak Harnack inequality (see Lemma \ref{hi1} below) and obtain that $u_2 \le C$ in $B_{3/2}$. This means that $u_1$ is a subsolution in the set $\{ u_1>C\} \cap B_{3/2}$, hence we apply Lemma \ref{hi1} one more time and bound $u_1$ by above in $B_1$. Similarly we bound $u_i$ by below and obtain
$$\|u_i\|_{L^\infty(B_1)} \le C.$$ 
Let 
\begin{equation}\label{vd}
v:= \chi_{B_1}u_2
\end{equation}
be the restriction of $u_2$ to $B_1$, and $x \in E \cap B_{1/2}$ (see \eqref{Es}). Then, since $v \le u_1$ in $B_1$, and $v(x)=u_1(x)$ we find
$$\mathcal L_{K_1} u_1 (x) \ge \mathcal L_{K_1} v(x) + \int_{\mathcal C B_1} (u_1(y)-v(x))\, K(y-x) dy,$$
hence
$$\mathcal L_{K_1} u_1(x)  \ge \mathcal L_{K_1} v(x) - C.$$

Moreover, for any $x \in B_{1/2}$ we have
$$|\mathcal L_{K_2}u_2(x)-\mathcal L_{K_2}v(x)| \le C,$$
in the viscosity sense. Combining the last two inequalities with the fact that $u_2$ is a subsolution and $(u_1,u_2)$ is a supersolution pair in the sense of Definition \ref{dvis} we obtain the following corollary.

\begin{cor}
The function $v$ defined in \eqref{vd} satisfies in $B_{1/2}$
\begin{equation}\label{v1}
\mathcal L_{K_2} v \ge -M, \end{equation}
\begin{equation}\label{v2}
\mathcal L_{K_2} v + \chi_E \mathcal L_{K_1} v \le M
\end{equation}
with $M$ a constant depending on $n$, $s_i$, $\lambda$, $\Lambda$.
\end{cor}

Inequality \eqref{v2} contains the discontinuous term $\chi_E$ and it is understood in the viscosity sense . Precisely, if $v$ admits a tangent $C^2$ function by below at a point $x$, then we satisfy two different inequalities depending whether or not $x$ is in $E$.   

 Since $s_2>s_1$ then the term  $\chi_E \mathcal L_{K_1} v$ can be treated as a perturbation. 
Then \eqref{v1}-\eqref{v2} can be thought heuristically as saying that $\mathcal L_{K_2}v \in L^\infty$, 
and we can infer that $v \in C^\beta$ for any $\beta <2 s_2$. We use the convention that when $\beta \in (1,2)$, the class $C^\beta$ denotes the class $C^{1,\beta-1}$. 
 We prove this statement rigorously in the next proposition.

\begin{prop}\label{l21} Assume $s_2>s_1$, and that $v$ is a continuous function supported in $B_1$ which satisfies  \eqref{v1}-\eqref{v2} for some closed set $E$. Then $v \in C^\beta$ for any $\beta<2s_2$ and
$$\|v\|_{C^\beta(B_{1/4})} \le C(\|v\|_{L^\infty}+M),$$
with $C$ a constant depending on $n$, $s_i$, $\lambda$, $\Lambda$ and $\beta$.
\end{prop}

\begin{proof}
The lemma can be deduced from the arguments of Caffarelli and Silvestre in \cite{CS2}. Since 
their results do not apply directly to our setting, we will sketch the proof of the proposition for 
completeness.

After multiplication by a small constant we may assume that $M=1$ and $\|v\|_{L^\infty(B_1)}$ is sufficiently small. 

We need to show that if for all balls $B_{r}$ with $r=2^{-l}, l=0,1, \ldots, k$ for some $k \ge k_0$, we have
\begin{equation}\label{lr}
|v-l_r| \le r^\beta \quad \mbox{in $B_r$,} 
\end{equation}
with $l_r$ a constant if $\beta<1$ or a linear function if $\beta>1$, and $l_1 \equiv 0$, 
then \eqref{lr} holds also in $B_{\rho r}$ for some $l_{\rho r}$ where $\rho=2^{-m_0}$. Here the constants $m_0$, $k_0$ depend on $\beta$ and the universal constants.
Then we can iterate \eqref{lr} indefinitely and obtain the desired conclusion.

The existence of $k_0$ is obtained by compactness. Indeed, assume that \eqref{lr} holds up to $r=r_k$ for some large $k$. 
Notice that the coefficients of $l_r$ are bounded by a fixed constant,
hence the rescaling $$\tilde v(x)=r^{-\beta}(v-l_r)(rx),$$
satisfies
$$\|\tilde v\|_{L^\infty(B_1)}\le 1, \quad \quad |\tilde v(x)| \le C_0 |x|^\beta \quad \mbox{outside $B_1$.}$$
Next we write \eqref{v1}-\eqref{v2} in terms of $\tilde v$. We have
$$\mathcal L_{K_2} v(x)=\mathcal L_{K_2} \left(l_r+r^\beta \tilde v(\frac x r) \right)=r^{\beta-2 s_2} \, \, \mathcal L_{\tilde K_2}\tilde v  \, (\frac x r).$$
We estimate $\mathcal L_{K_1} v$ by writing
$$v(x)=\chi_{B_1}v(x)=\chi_{B_1}l_r(x) + \chi_{B_1\setminus B_{2r}} \, \, r^\beta \tilde v(\frac x r) + \chi_{B_{2r}} \, \, r^\beta \tilde v(\frac x r) =:v_1+v_2+v_3.$$
We have $|\mathcal L_{K_1} v_1| \le C$ in $B_r$.
Without loss of generality we may assume that $\beta>2 s_1$ which, by the growth of $\tilde v$ outside $B_1$ gives
$|\mathcal L_{K_1} v_2| \le C$ in $B_r$. 
Also 
$$\mathcal L_{K_1} v_3(x)=r^{\beta-2s_1} \, \mathcal L_{ \tilde K_1} (\chi_{B_2} \tilde v) (\frac x r).$$
In conclusion $\tilde v$ satisfies in $B_1$ the following inequalities
\be \label{k2v1}
\mathcal L_{\tilde K_2} \tilde v \ge -Cr^{2s_2-\beta},
\ee
\be\label{k2v2}
\mathcal L_{\tilde K_2} \tilde v + r^{2(s_2-s_1)} \, \, \chi_{\tilde E} \, 
\mathcal L_{\tilde K_1} (\chi_{B_2} \tilde v) \le C r^{2s_2-\beta}.
\ee
The function $\tilde v$ is both a subsolution and a supersolution for integro-differential equations 
with measurable kernels and bounded right hand side. Since $r^{2(s_2-s_1)}$ is small, the 
two operators above are bounded by two extremal Pucci operators of order $2 s_2$. We
apply the Harnack inequality for integro-differential equations from \cite{CS1} and obtain that 
$\tilde v$ is uniformly H\"older continuous in $B_{3/4}$. 
This means that as $r \to 0$ (or equivalently as $k \to \infty$), the corresponding $\tilde v$'s 
converge uniformly on a subsequence  
to a limit $\bar v$. We claim that $\bar v$ satisfies
$$|\mathcal L_{\bar K}(\chi_{B_{3/4}} \bar v)| \le C \quad \mbox{ in $B_{1/2}$},$$ 
where $\bar K$ is the weak limit of the $\tilde K_2$'s.

Indeed, let $\tilde w:= \chi_{B_{3/4}} \tilde v $, then \eqref{k2v1}-\eqref{k2v2} give
$$\mathcal L_{\tilde K_2} \tilde w \ge -C,  \quad \quad \quad \mathcal L_{\tilde K_2} \tilde w + r^{2(s_2-s_1)} \, \, \chi_{\tilde E} \, 
\mathcal L_{\tilde K_1} \tilde w \le C,$$
with $r^{2(s_2-s_1)} \to 0$. Now we can pass to the limit in these inequalities and use that
$\mathcal L_{\tilde K_2}\psi (x) \to \mathcal L_{ \bar K} \psi(x) $ for any test function $\psi \in C^2$ near $x$, and obtain the claim.

The existence of $l_{\rho r}$ with $\rho=2^{-m_0}$ universal, now follows from the $C^{\beta+\eps}$ estimates, with $\beta + \eps<2s_2$, of the solution $\bar v$ above, see Proposition \ref{SE}, part $a)$. 
\end{proof}

\begin{rem}
We are not concerned in obtaining estimates that remain uniform as the order of 
the operators approaches 2.

The Harnack inequality for $\tilde v$ can be checked also directly by using the methods of 
Silvestre in \cite{S2}. For this we slide parabolas by above and below till they touch the graph of $
\tilde v$. Then we use the equation only at these points and show that the oscillation of $\tilde v$ 
decays at a geometric rate as we restrict to dyadic balls. We will use this method more precisely in 
Section 5, see Step 1 in Proposition \ref{po1}.

We remark the same argument works as well in the case when $\mathcal L_{K_2}$ is a local operator, 
and then we need to use the ABP measure estimate, see \cite{Sa} for example.

\end{rem}

Proposition \ref{l21} provides the initial $C^{2 s_2 -\eps}$ interior regularity of the function $u_2$. 
Now we can view the function $u_1$ as the solution to the obstacle problem with obstacle $u_2$.
Therefore in our analysis it is important to obtain regularity of solutions to the obstacle 
problem with not necessarily $C^2$ obstacle. 
In the next two sections we show that $u_1$ is as regular as the obstacle up to 
$C^{\max \{1,2s_1 \} + \eps}$ regularity in the case of translation invariant kernels, 
and up to $C^{1+s_1}$-regularity in the case of the fractional Laplacian. 

Then we can successively improve the regularity of $u_2$ and $u_1$ and obtain Theorems \ref{T2} and \ref{T3}. 

\medskip

{\it Proof of Theorem \ref{T2}.}  
From Theorem \ref{T20} in Section 5 we have that $u_1$ is as regular as $u_2$ up to $C^{\max \{1,2s_1 \} + \eps}$ regularity, and $u_2 \in C^{2s_2-\eps}$ by Proposition \ref{l21}. From the Schauder estimates for the equation $\mathcal L_Ku=f$, see Proposition \ref{SE} in the Appendix, this implies that $\mathcal L_{K_1} u_1 \in C^\eps$. Thus $\mathcal L_{K_2} u_2 \in C^\eps$ which gives
$u_2 \in C^{2s_2+\eps}$. Now we can iterate this argument and obtain the desired conclusion.
\qed

\section{The obstacle problem for translation invariant kernels}

In this section we make a detour to provide two regularity results for the general obstacle problem in the case of symmetric, translation invariant operators 
$\mathcal L_K$ as above. We then apply
these results to the two membranes problem.

In addition to \eqref{Kc} we need to impose the extra regularity assumption on $K$, i.e.
\begin{equation}\label{Kc1}
|\nabla K(y)| \le \Lambda |y|^{-(n+1+2s)}.
\end{equation}

Assume that $u$ is a solution of the obstacle problem in $B_1$ with obstacle $\varphi$ by below. 
Precisely we assume that $u$, $\varphi$ are continuous in $B_1$, $u \in L^1(\R^n, d \omega)$, 
and
\begin{equation}\label{opti1}
u \ge \varphi \quad \mbox{in $B_1$} ,
\end{equation}
\begin{equation}\label{opti2}
\mathcal L_{K} u \le f \quad \mbox{in $B_1$}, \quad \mbox{and} \quad \quad \quad \mathcal L_{K} u = f \quad \mbox{in} \quad \{u>\varphi\} \cap B_1.  
\end{equation}

Our main result of this section says that up to $C^{1,\eps_0 }$ with $\eps_0$ universal, the solution $u$ is as regular as the obstacle $\varphi$.
Moreover,  in the case $s > \frac 12$, the  $C^{1,\eps_0}$ regularity can be improved to $C^{2s+\eps_0}$. 
\begin{thm}\label{T20}
Let $u$ is a solution to the obstacle problem \eqref{opti1}, \eqref{opti2}, with kernel $K$ that satisfies \eqref{Kc},\eqref{Kc1}, and assume that
$$\|u\|_{L^1(\R^n, d \omega)}, \|\varphi\|_{C^\beta(B_1)}, \|f\|_{C^{0,1}(B_1)} \le 1.$$
for some $\beta \ne 2s$. 

Then $u \in C^\alpha(B_1)$ for $\alpha = \min \{\beta, \max \{1, 2s\}+ \eps_0 \}$ and
 $$\|u\|_ {C^\alpha(B_{1/2})} \le C,$$
where $\eps_0$ depends on $n$, $\lambda$, $\Lambda$, $s$, and the constant $C$ may depend also on $\beta$.
\end{thm}

Before we proceed with the proof of Theorem \ref{T20} we write two versions of Harnack inequality for nonlocal equations which deal with $L^\infty$ bounds for subsolutions.

\begin{lem}\label{hi1}
Assume that $v$ is continuous in $\overline{B}_1$, $\| v^+\|_{L^1(\R^n,d \omega)} \le 1$, and
$$\mathcal L_K v \ge -1 \quad \quad \mbox{in} \quad \{v>1\} \cap B_1.$$
Then $v \le C$ in $B_{1/2}$ with $C$ depending only on $n$, $s$, $\lambda$, $\Lambda$.
\end{lem}

\begin{proof}
After multiplication with a small constant we may replace 1 by $\delta_0$ in the hypotheses above.  We show that $v \le \psi$ with
$$\psi(x):=(1-|x|^2)^{- n}.$$
Assume by contradiction that when we slide the graph of $\psi$ by above we touch the graph of $v$ at some point $(x_0,v(x_0))$ above the original graph of $\psi$, i.e.
there exists $t>0$ such that $v \le \psi_t$ in $B_1$ and $v(x_0)=\psi_t(x_0)$ for some $x_0$, where $\psi_t:=\psi+t$. 
Denote by $$d:=1-|x_0|,$$ and by $l$ the tangent plane of $\psi_t$ at $x_0$.
Then for $r \le d/2$ we have
$$\int_{B_{r}(x_0)} (v(x)-v(x_0)) K(x-x_0) dx \le  \int_{B_{r}(x_0)} \left(\Lambda (v -l)^+ - \lambda (v-l)^-\right) |x-x_0|^{-n-2s} dx $$
$$ \le  C d^{-n-2}r^{2-2s} - \lambda r^{-n-2s} \int_{B_r(x_0)}(v-l)^- dx.$$
We use $$\int_{B_r(x_0)}(v-l)^- dx \ge \int_{B_r(x_0)}(l-v) dx \ge \psi_t(x_0) |B_r| - \int_{B_1}v^+ dx  $$
$$\ge C d^{-n} r^n - \delta_0,$$
which, by taking $r=dc$ with $ c$ small, and $\delta_0 \ll c$ sufficiently small, we obtain
$$ \int_{B_{r}(x_0)} (v(x)-v(x_0)) K(x-x_0) dx \le - c r^{-n-2s}.$$
On the other hand
$$ \int_{\mathcal C B_{r}(x_0)} (v(x)-v(x_0)) K(x-x_0) dx \le \Lambda \int_{\mathcal C B_{r}(x_0)} v^+(x) |x-x_0|^{-n-2s} dx \le C \delta_0 r^{-n-2s}.$$
From the last two inequalities we find
$$\mathcal L_K v(x_0) \le -c,$$
and we reached a contradiction, provided that $\delta_0$ is chosen sufficiently small.
\end{proof}

We remark that in the proof we did not use the translation invariant properties of $K$, and clearly the proof holds for truncated kernels $ \chi_{B_2}K$ as well.
Also the assumption on the bound for the $L^1$ norm of $v^+$ in $\R^n$ can be weakened to an $L^1$ bound for $v^+$ only on $\mathcal CB_{3/4}$. This can be seen by appropriately modifying the comparison function $\psi$ in the proof.

We provide a version of Harnack inequality that follows from Lemma \ref{hi1}.
\begin{lem}\label{hi2}
Assume that $v \ge 0$ in $B_1$, $v(0) \le 1$, 
$$\mathcal L_K v\le \sigma \quad \mbox{in $B_1$}, \quad \mathcal L_Kv \ge \sigma -1 \quad \mbox{in $\{v > 1 \} \cap B_1$,}$$
for some $\sigma$, and $$\int |v| (\max\{1,|x|\})^{-(n+1+2s)} \,dx \le  1.$$
Then $v \le C$ in $B_{1/2}$ with $C$ independent of $\sigma$.
\end{lem}

\begin{proof} Let $K_T= \chi_{B_2} K$ be the truncation of $K$, and
we show that $v$ and $K_T$ satisfy the hypotheses of Lemma \ref{hi1}. 
We slide the parabola $x_{n+1}=-4 |x|^2$ by below till it touches the graph of $v$ at some point $y_0$, 
and from our hypotheses above it follows that  $y_0 \in B_{1/2}$, $v(y_0) \le 1$, and
$$\mathcal L_{K_T} v(y_0) \ge -C.$$
%Moreover since $ \mathcal L_{K_T} v(x_1) \ge -C  $
 For $y \in B_1$ we have
$$ \mathcal L_{K} v(y)- \mathcal L_{K} v(y_0) \le  \mathcal L_{K_T} v (y) - \mathcal  L_{K_{T}}v(y_0) + \int_{\mathcal C B_2} v(x) (K(x-y)-K(x-y_0)) dx + C,$$
and from \eqref{Kc1} we have that $$|K(x-y)-K(x-y_0)| \le C |x|^{-(n+1+2s)}  \quad \quad \mbox{if} \quad x \in \mathcal C B_2.$$
Thus
$$\mathcal L_{K_T} v (y) \ge -C \quad \mbox{in $ \{v>1\} \cap B_1$},$$
and the conclusion follows from Lemma \ref{hi1}.
\end{proof}

\begin{rem} \label{r1}We remark that if we slide a parabola $4C |x|^2$ by above and it touches the graph of $v$ at some point $y_1$ for which $\mathcal L_K v(y_1) \ge \sigma -1$ then
by repeating the argument ``upside-down" (i.e. for $-v$) we obtain $ \mathcal L_{K_T} v (y)  \le C $ in $B_1$. 
\end{rem}

We are now ready to prove Theorem \ref{T20}, which is a direct consequence of Propositions \ref{po1} and \ref{po2} below. First we state the necessary Schauder estimates, which will be proved in the appendix.

\begin{prop}[Schauder estimates]\label{SE5}
Let $K$ be a symmetric kernel that satisfies \eqref{Kc}, and assume that $v \in L^1(\R^n,d\omega)$ satisfies
$$\mathcal L _K v=f \quad \mbox{in $B_1$}, \quad \quad \quad \|v\|_{L^\infty(B_1)} \le 1.$$

a) If $\|f\|_{L^\infty (B_1)} \le 1 $, $\|v\|_{L^1(\R^n,d \omega)} \le 1$ then $$\|v\|_{C^\alpha(B_{1/2})} \le C(\alpha), \quad \mbox{for any $\alpha < 2s$.}$$

b) Assume that $K$ satisfies \eqref{Kc1}. If $$ \int_{\mathcal C B_1} v  \, \, |x|^{-(n+2s+1)} \, dx \le 1, \quad [f]_{C^\gamma (B_1)} \le 1, \quad \mbox{ for some $\gamma \in (0,1)$}$$
 then $$\|v\|_{C^\beta(B_{1/2})} \le C(\gamma), \quad \quad \mbox{with $\beta=2s + \gamma,$}$$
provided that $2s+\gamma$ is not an integer.

c) Conversely, if $\|v\|_{L^1(\R^n,d \omega)} \le 1$, $\|v\|_{C^\beta(B_1)} \le 1$ with $\beta$ as above, then $$\|f\|_{C^\gamma(B_{1/2})} \le C.$$
\end{prop}

Proposition \ref{SE5} can be easily deduced from the results of Serra in \cite{Ser} where he obtained Schauder estimates for concave integro-differential equations with rough kernels (see also \cite{K},\cite{CS2}). We will sketch the proof in the Appendix, since its statement is slightly different than it usually appears in the literature and our setting is simpler than in \cite{Ser}.

Next, we prove the statement in Theorem \ref{T20}, valid for all $s \in (0,1),$ that is the following proposition.

\begin{prop}\label{po1}
Let $u$ satisfy \eqref{opti1}, \eqref{opti2} and assume that
$$\|u\|_{L^1(\R^n, d \omega)}, \|\varphi\|_{C^\beta(B_1)}, \|f\|_{C^{0,1}(B_1)} \le 1.$$
Then $u \in C^\alpha(B_1)$ for $\alpha = \min \{\beta, 1+ \eps_0 \}$ and $\|u\|_ {C^\alpha(B_{1/2})} \le C.$
%where $C$, $\eps_0$ are constants depending on $n$, $\lambda$, $\Lambda$, $s$, $\beta$.
\end{prop}

\begin{proof}
We sketch the proof below. In view of Lemma \ref{hi1}, we can assume without loss of generality that $\|u\|_{L^\infty} \leq 1$ in $B_1.$ In fact, 
after multiplication with a small constant, we may assume that all the norms in our assumptions and $\|u\|_{L^\infty}$ are bounded by $\delta_0$, sufficiently small to be made precise later.

\smallskip

{\it Step 1:} We show that $u \in C^{\alpha_0}$ for a small $\alpha_0>0$, 
by checking that the usual proof for H\"older continuity of solutions to nonlocal equations \cite{S2} still applies in our case. 
Let us assume for simplicity that $0 \in \{u=\varphi\}$, $u(0)=0$ and suppose that 
\begin{equation}\label{i1}
u \le r^{\alpha_0}=(1-\delta)^l \quad \mbox{in $B_r$, with} \quad r=2^{-l}, \quad \mbox{ for all $l \le k$},
\end{equation}
 for some $k \ge k_0$. Then we need to show that \eqref{i1} holds for $l=k+1$ as well.

Indeed, the rescaling $\tilde u(x):=r^{-\alpha_0}u(rx)$ with $r=2^{-k}$ satisfies in $B_1$ ($\alpha_0 \leq \beta$)
$$-\delta_0 \le \tilde u \le 1, \quad \mathcal L_{\tilde K} \tilde u \le \delta_0, \quad \quad \mathcal L_{\tilde K} \tilde u \ge -\delta_0 \mbox{ in $\{\tilde u > \delta_0\}$}.$$
Moreover,
\begin{eqnarray}\label{levelk}
\tilde u \leq (1-\delta)^j, \quad \text{in $B_{2^j}, \quad j=1,\ldots, k,$}\\
\int_{\R^n \setminus B_{2^k}} \tilde ud\omega \leq (2^{-k})^{2s-\alpha_0}\delta_0. 
\end{eqnarray}
In order to obtain the diminish of oscillation of $\tilde u$ 
we compute $\mathcal L_K \tilde u$ at the two contact points $x_0^-$, $x_0^+$ 
obtained by sliding two paraboloids of opening $2\delta$ by below and above till they touch the graph of $\tilde u$. 

Precisely, we slide $P_t:=2\delta|x|^2+t$, $t \leq 1$, from above. If no contact point occurs till $ t=1-\frac 3 2 \delta,$ then $$\tilde u \leq 1-\delta \quad \text{in $B_{1/2}$}$$
and we obtain the desired diminish in oscillation. Let us consider then the case when the contact point $x_0^+$ occurs for $t > 1-3/2\delta$, that is near the top $x_{n+1}=1$. Hence (say $\delta_0<1/4, \delta <1/2$)$$u(x_0^+)> \delta_0 \quad \text{and} \quad \mathcal L_{\tilde K} \tilde u (x_0^+) \geq -\delta_0.$$ 
Assume that 
\begin{equation}\label{measure12}|\{\tilde u > \frac 1 2\} \cap B_1| < \frac 1 2 |B_1|.
\end{equation}
We show that
\begin{equation}\label{boundc}\mathcal L_{\tilde K}\tilde u(x_0^+) \leq -c\end{equation}
for $c$ universal, provided that $\delta$ (hence $\alpha_0$) is small enough. We thus reach a contradiction if $\delta_0$ is small enough.

Indeed, for $\delta$ small,
$$\tilde u \leq P_t - \frac 1 4 \chi_{\{\tilde u \leq \frac 12 \}} \quad \text{in $B_1$}.$$
Hence,
\begin{eqnarray*}
\mathcal L_{\tilde K} \tilde u(x_0^+)  \leq \int_{B_1} (P_t(x)-P_t(x_0^+))\tilde K(x-x_0^+)dx -\frac 1 4\int_{\{\tilde u \leq \frac 1 2\} \cap B_1} \tilde K(x-x_0^+) \\
+ \int_{\R^n \setminus B_1} (\tilde u(x)-\tilde u(x_0^+))\tilde K(x-x_0^+) dx := I_1+I_2+I_3. 
\end{eqnarray*}
We first observe that $x_0^+ \in B_{3/4}$, since $\tilde u\leq 1$ and $t > 1-3/2\delta.$

It is easily seen that
$$I_1 \leq C_1\delta.$$ Moreover, from \eqref{measure12} we have
$$I_2 \leq - c_2.$$ 
Finally, we estimate $I_3$ as follows, and we recall that $k \geq k_0$ large.
$$I_3 \leq \sum_{j=1}^{k} \int_{B_{2^{j}} \setminus B_{2^{j-1}}} (\tilde u(x)-\tilde u(x_0^+))\tilde K(x-x_0^+) dx + \int_{\R^n \setminus B_{2^{k}}} \tilde u d\omega=I_3^1+I_3^2.$$
To estimate $I_3^1$ we use \eqref{levelk} and get
$$I_3^1 \leq C \sum_{j=1}^k ((1-\delta)^{-j}- 1+\frac 3 2 \delta)2^{-2sj} \leq c(\delta) \to 0, \quad \text{as $\delta \to 0.$}$$
Again, to estimate $I_3^2$ we use \eqref{levelk} and obtain
$$I_3^2 \leq (2^{-k})^{2s-\alpha_0} \to 0 \quad \text{for $k_0$ large enough and $\delta$ (hence $\alpha_0$) small.}$$
Combining the estimates above, we obtain the claim in \eqref{boundc} and reach a contradiction.

This implies that either the contact point does not occur near the top, and we are done, or \eqref{measure12} does not hold and
\begin{equation}
|\{\tilde u > \frac 1 2\}\cap B_1| \geq \frac 1 2 |B_1|.
\end{equation}
In this case, we slide $-2\delta|x^2|-t$ by below, $t \geq \delta_0$, and we work with the lower contact point $x_0^-.$ Since $\tilde u(0)=0$ we see that $x_0^-$ occurs close to the bottom $x_{n+1}=-\delta_0.$
With a similar computation as above, we obtain that
$$\mathcal L_{\tilde K} \tilde u(x_0^-) \geq c,$$ with $c$ universal ($\delta$ chosen small).  This contradicts that $\mathcal L_{\tilde K} \tilde u(x_0^-) \leq \delta_0$, if $\delta_0$ is small.
This means that \eqref{measure12} must hold and $x_0^+$ will occur far from the top, providing the diminish in the oscillation.

This establishes a uniform pointwise $C^{\alpha_0}$- Holder continuity of $u$ at all points on the contact set $\{u=\varphi\} \cap B_{1/2}$. 
It is easy to extend this modulus of continuity at all $x \in B_{1/4}$. 
We take the largest ball $B_\rho (x)$ included in $\{u>\varphi\}$ which is tangent to $\{u=\varphi\}$ at some point $y$, 
and then we apply the interior estimates in Proposition \ref{SE5} to $\mathcal L_Ku=f$ in $B_\rho(x)$ by using the modulus of continuity of $u$ at $y$. 

\

{\it Step 2:} We show that if  $u \in C^{\alpha}$ for some $\alpha \leq 1$ then 
$u \in C^{\alpha + \eps_0}$ for some $\eps_0$ universal, as long as $\alpha+\eps_0 \leq \beta$. 
Then we combine this claim and step 1, and obtain the desired conclusion.

The proof is similar to the one in Step 1, and uses the fact that the derivatives of $u$ are ``subsolutions". Let us assume that the norms of the data are bounded by $\delta_0$ and that
$$u(0)=\varphi(0)=0, \quad  \quad \nabla \varphi (0)=0 \quad
\mbox{if $\beta>1$}, \quad \mbox{and} \quad \|u\|_{C^\alpha(B_1)} \le \delta_0.$$
We consider the difference quotients
$$u_h^e(x):=\frac{u(x+he)-u(x)}{h^\alpha} ,$$
where $e$ is a unit vector and prove the following property.

Assume that for some $k \ge k_0$, we have for all $r=2^{-l}$ with $l \le k$ 
\begin{equation}\label{i2}
u_h^e \le r^{\eps_0}=(1-\delta)^l \quad \mbox{in $B_r$, for all $h \le r$, $|e|=1$.}
\end{equation}
Then \eqref{i2} holds for $l=k+1$ as well.

Fix $r=2^{-k}$. The key observation is that
\begin{equation}\label{ko}
\mathcal L_K u_h^e \ge f_h^e \ge -\delta_0 \quad \quad \mbox{in} \quad \{u_h^e>\frac 12 r^{\eps_0} \} \cap B_r.
\end{equation}
Indeed, since $u$ is a solution in the set $\{u>\varphi\}$ and a supersolution in $B_1$, 
we conclude that the only points where the inequality in \eqref{ko} can fail are those with $x+he \in \{u=\varphi\}$. 
At these points
$$u_h^e(x) \le \varphi_h^e(x) \le \quad \delta_0 h^{\beta-\alpha}
 \quad \mbox{(or} \quad \delta_0 r^{\beta-1} h^{1-\alpha}\quad 
\mbox{if $\beta>1$)} \quad \le \frac 12 r^{\eps_0}.$$

Moreover, call $K_T=\chi_{B_{1/4}}K$, then for a universal $c>0,$
\begin{equation}\label{koT}
\mathcal L_{K_T} u_h^e \ge-c \quad \quad \mbox{in} \quad \{u_h^e>\frac 12 r^{\eps_0} \} \cap B_r.
\end{equation}
Indeed for $x$ in such set $u_h^e(x)>0$ and we have,
$$\mathcal L_{K_T} u_h^e \ge -\delta_0 - \int_{\mathcal C B_{1/4}(x)} u_h^e(y)K(y-x)dy.$$
Call the second term $E.$ Then, one easily sees that
$$|E| \leq \frac{1}{h^\alpha} (E_1+E_2+E_3),$$
with
\begin{align*}
& E_1:= \int_{A_1} |u(x+z)||K(z)-K(z-he)| dz, \quad A_1=\mathcal C (B_{1/4} \cup B_{1/4}(he));\\
&E_2:= \int_{A_2} |u(x+z)|K(z-he)dz, \quad A_2:=B_{1/4} \setminus B_{1/4}(he);\\
&E_3:=\int_{A_3} |u(x+z)|K(z)dz, \quad A_3:=B_{1/4}(he) \setminus B_{1/4}.
\end{align*}
Since $h \leq r=2^{-k}$ with $k$ large, and $u$ is bounded in $B_1$, then $E_2,E_3 \leq C h.$ To bound $E_1$ we use that $\|u\|_{L^1(d\omega)} \leq \delta_0$ and assumption \eqref{Kc1}. We thus obtain $E_3\leq Ch$ as well and by collecting all these bounds we obtain the desired claim.

Now, let $$\tilde u(x):=r^{-(\alpha+\eps_0)}u(rx),$$ be the rescaling of $u$ and notice that from $u \ge \varphi$ and \eqref{i2} applied with $x=0$, $he=r y$, $y\in B_1$ we find 
\begin{equation}\label{tui}
-\delta_0 \le \tilde u(y) \le |y|^\alpha \quad \mbox{in $B_1$}.
\end{equation}

Let $h \le r/2$, and write $h=r \tilde h$, with $\tilde h \le 1/2$. Then
$$v(x):=\tilde u_{\tilde h}^e(x)=r^{-\eps_0}u_h^e(rx),$$
is the rescaling of $u_h^e$ from $B_r$ to the unit ball, and from \eqref{i2},\eqref{koT} in $B_1$ we obtain that in $B_1$
$$-2 \le v \le 1, \quad \quad \mathcal L_{\tilde K_T}  v \ge -\delta_0 \mbox{ in $\{ v > \frac 12\}$},$$
where the lower bound on $v$ follows from \eqref{i2} applied for $-e$. Here
$$\tilde K_T=\chi_{B_{1/4r}}K.$$

Now we claim that $|\{v<1-c\} \cap B_1| \ge c$ for some fixed $c$ small universal.
The reason is that if $v$ is close to $1$ in almost all $B_1$ then we contradict that $\tilde u \ge -\delta_0$. Indeed, assume for simplicity that $e=e_n$ and we integrate $v$ in the cylinder
$$\mathcal C: = \left \{|x'| \le \frac 18, \quad x_n \in [-\frac 34, \frac 14] \right\} .$$ 
For each segment in the $e_n$ direction $l_{x'}=\{(x',x_n)| x_n \in [-\frac 34,\frac 14]\}$ of length 1 included in $\mathcal C$ we have (see \eqref{tui})
$$\int_{l_{x'}}v \, dx_n = \tilde h^{-\alpha} \left(\int_{\frac 14}^{\frac 14 + \tilde h} \tilde u \, \, dx_n- \int_{- \frac 34}^{- \frac 3 4+\tilde h} \tilde u \, \, dx_n \right) $$
$$\le \tilde h^{1-\alpha} \left( (\frac 78)^\alpha  + \delta_0  \right) \le 1-c,  $$
and our claim follows.

Now the proof of diminish of oscillation for $v$ follows as in Step 1. We remark that in bounding $\mathcal L_{\tilde K_T}\tilde v$ at the contact point, we will not have a term as $I_3^2$, since the kernel $\tilde K$ is truncated. All the other terms can be bounded with similar arguments as above.

In conclusion property \eqref{i2} is proved and this implies that $u \le r^{\alpha + \eps_0}$ in $B_r$ for all dyadic balls, thus $u$ is pointwise $C^{\alpha+\eps_0}$ at $0$. 
Now we can extend as above the pointwise regularity from the set $\{u=\varphi\}$ to the whole $B_{1/4}$, and obtain the desired conclusion.
\end{proof}

We show that when $s> \frac 12$, then the result of Proposition \ref{po1} can be improved.
\begin{prop}\label{po2}
Let $u$ satisfy \eqref{opti1}, \eqref{opti2} and assume $s>1/2,$
$$\|u\|_{L^1(\R^n, d \omega)}, \|\varphi\|_{C^\beta(B_1)}, \| f \|_{C^{\eps_0}(B_1)} \le 1,$$ for some $\beta \ne 2s$. 
Then $u \in C^\alpha(B_1)$ for $\alpha = \min \{\beta, 2s+ \eps_0 \}$ with
 $$\|u\|_ {C^\alpha(B_{1/2})} \le C.$$
\end{prop}

\begin{proof} Assume that $\| u \|_{L^1(d \omega)}$, $\|\varphi\|_{C^\beta}$, $ \|f\|_{C^{\eps_0}} $ are all smaller than $\delta_0$, 
and assume also that $u(0)= \varphi(0)=0$, and $\nabla \varphi(0)=0$ if $\beta>1$. We treat the case when $\beta \ge 2s + \eps_0$.

We prove by induction that there exists a sequence of radii $1=r_1>r_2 >....$ with $r_{k+1}/r_k \in[\rho_0,1/2)$ for some fixed $\rho_0$ such that
\begin{equation}\label{ihpo2}
\int  |u| (\max \{  r,|x|\})^{-(n+1+2s)} dx \le r^{\eps_0-1} .
\end{equation}
Assume that this holds for some $r=r_k$. We let $$\tilde u(x)= r^{-2s-\eps_0} u(rx), \quad \tilde \varphi(x)= r^{-2s-\eps_0} \varphi(rx), \quad \tilde f(x)=r^{-\eps_0}f(rx),$$ and we have
$$\mathcal L_{\tilde K} \tilde u \le \tilde f  \quad \mbox{in $B_1$}, \quad \quad  
\mathcal L_{\tilde K} \tilde u = \tilde f \quad \mbox{in $\{\tilde u > \tilde \varphi  \} \cap B_1$},$$
and $$ osc_{B_1} \, \,  \tilde f \le \delta_0, \quad |\tilde \varphi(x)| \le \delta_0 |x|^{2s+\eps_0}  \quad \quad \mbox{in $B_1$}.$$
Moreover, \eqref{ihpo2} is equivalent to
\begin{equation}\label{eq20}
\int  |\tilde u| (\max \{  1,|x|\})^{-(n+1+2s)} dx \le 1.
\end{equation}
We want to show that there exists $\rho \in [\rho_0,\frac 12)$ such that
\begin{equation}\label{21}
\int  |\tilde u| (\max \{  \rho,|x|\})^{-(n+1+2s)} dx \le \rho^{\eps_0-1},
\end{equation}
and then the induction hypothesis \eqref{ihpo2} is satisfied for $r_{k+1}=\rho r_k$.

Notice that $\tilde u + \delta_0$ satisfies the hypotheses of the Lemma \ref{hi2} hence $\tilde u \le C$ in $B_{1/2}$. Now we distinguish two cases.

{\it Case 1:} $\tilde u \le \delta_0$ in $B_{1/4}$. Then \eqref{21} is satisfied clearly satisfied for $\rho=\rho_0$ small, provided that $\delta_0 \ll \rho_0$ is chosen sufficiently small.

{\it Case 2:} $\tilde u > \delta_0$ for some point in $B_{1/4}$. The according to Remark \ref{r1} we can slide a parabola of fixed opening by above and obtain a contact point in $\{\tilde u > \delta_0 > \tilde \varphi\}$ thus  $$\mathcal L_{\tilde K_T} \tilde u(0) \le C.$$  
Since $\tilde \varphi$ is tangent by below to $\tilde u$ at $0$ the above inequality implies
\begin{equation}\label{eq22}
\int_{B_1} |\tilde u| |x|^{-n-2s} dx \le C.
\end{equation}
On the other hand, if we assume by contradiction that \eqref{21} holds in the opposite direction for all $\rho \in (\rho_0,1/2)$ then we can integrate this inequality in $\rho$ and obtain
$$\int  |\tilde u| (\min \{ 1,|x|\}) |x|^{-(n+1+2s)} dx \ge \eta(\rho_0,\eps_0),$$
with $\eta(\rho_0,\eps_0) \to \infty$ as $\rho_0, \eps_0 \to 0$. This contradicts \eqref{eq20}, \eqref{eq22} by choosing $\eps_0$, $\rho_0$ sufficiently small.

In conclusion property \eqref{ihpo2} is proved, and from the argument above we obtain $u(x) \le C |x|^{2s+\eps_0}$ in $B_1$. This means that $u$ is pointwise $C^{2s+\eps_0}$ in the set $\{u= \varphi\}$, and this can be extended to the whole $B_{1/2}$ as before.

When $\beta \in (2s,2s+\eps_0)$ the argument above applies with $\eps_0$ replaced by $\beta-2s$. 

Finally, when $\beta<2s$ the proof is simpler. The rescaling 
$\tilde u(x)=r^{-\beta} u(rx)$ satisfies $\|\tilde u\|_{L^1(\R^n, d\omega)} \le C$, (since now $\tilde \varphi$ is integrable at infinity) and we can apply Lemma \ref{hi1} directly to obtain the pointwise $C^\beta$ estimate at the origin. In this case we only require $f \in L^\infty$. 
\end{proof}

\section{The case of the fractional Laplacian: free boundary regularity}

In the special case when 
$$K_1(y)=\frac{1}{|y|^{n+2s_1}}$$
the operator $\mathcal L_{K_1}$ 
is the fractional Laplacian $\Delta^{s_1}$ and we obtain the optimal regularity of the 
minimizing pair in the two membranes problem, see Theorem \ref{T3}. This improvement is due to 
the fact that the optimal $C^{1,s}$ regularity in the obstacle problem for the fractional Laplacian is 
known. Precisely, 
assume that $u$ is a solution of the thin obstacle problem in $B_1$ with obstacle $\varphi$ by below, that is 
$u$, $\varphi$ are continuous in $B_1$, $u \in L^1(\R^n, d \omega)$, 
and
\begin{equation}\label{opti3}
u \ge \varphi \quad \mbox{in $B_1$} ,
\end{equation}
\begin{equation}\label{opti4}
\mathcal \triangle^s u \le f \quad \mbox{in $B_1$}, \quad \mbox{and} \quad \quad \quad \mathcal \triangle^s u = f \quad \mbox{in} \quad \{u>\varphi\} \cap B_1.  
\end{equation}

The following result holds (see Section 1 for the notion of regular points.)

\begin{thm}[Optimal regularity]\label{T31}
Let $u$ be a solution to \eqref{opti3}, \eqref{opti4}, with
$$\|u\|_{L^1(\R^n, d \omega)}, \|\varphi\|_{C^\beta(B_1)}, \|f\|_{C^{\beta-2s}(B_1)} \le 1, \quad \quad \mbox{for some $\beta > 1 + s$.}$$
Then $u \in C^{1+s}(B_1)$ and
 $$\|u\|_ {C^{1+s}(B_{1/2})} \le C.$$
 Moreover, the free boundary $\Gamma:=\p \{u=\varphi  \}$ is a $C^{1,\gamma}$ surface in a 
 neighborhood of each of its regular points. The constants $C, \gamma$ 
 depend on $n$, $s$, and $\beta$.
 
\end{thm}

Theorem \ref{T31} was obtained by Caffarelli, Salsa and Silvestre in [CSS] . The main tool in the proof is to establish a version of 
Almgren's frequency formula for the ``extension" of $u$ to $\R^{n+1}$. Theorem \ref{T31} is proved in \cite{CSS} in the case when $\varphi \in C^{2,1}$ ( i.e. $\beta=3$). Below we 
show that the Almgren's monotonicity formula still holds when $\beta>1+s$. Since this is the only 
place in the proof in \cite{CSS} where the regularity of the data is needed, we obtain the version of 
Theorem \ref{T31} above.

Finally we remark that in the case when $\beta \in (2s,1+s)$ the $C^{1,\alpha}$ regularity of $u$ 
with $\alpha < \beta$ was obtained by Silvestre in \cite{S2}.

\subsection{Almgren's monotonicity formula} In this section, $\mathcal B_r$ will denote a ball in $\R^{n+1}$ and $B_r:= \mathcal B_r \cap \{x_{n+1}=0\}.$ Also, $X=(x,x_{n+1})$ is a point in $\R^{n+1}$ and often we call $y=x_{n+1}.$

After subtracting an explicit function whose fractional Laplacian equals $f$, we may assume 
without loss of generality that $f=0$. 
Let $u$ be a solution in $B_2$ to the thin obstacle problem 
\begin{align}\label{TO}
\nonumber & u \geq \varphi \quad \quad \quad \text{in $B_2\subset \R^n$}\\
& \Delta^s u = 0 \quad \text{in $\{u>\varphi\}\cap B_2$}\\
\nonumber & \Delta^s u \leq 0 \quad \text{in $B_2$}
\end{align}
with $\varphi: B_2 \to \R$ a continuos function.

Consider the equivalent (localized) problem obtained extending $u$ to $\R^{n+1}$, evenly in the $y=x_{n+1}$ direction,
\begin{align*}
& u(x,0) \geq \varphi \quad \text{for $x\in B_2$}\\
& u(x,y)=u(x,-y)\\
& L_a u=\text{div}(|y|^a\nabla u(x,y))=0 \quad \text{in $\mathcal B_2 \setminus \{u(x,0)=\varphi(x)\}$}\\
& L_a u \leq 0 \quad \text{in $\mathcal B_2$ in the distributional sense}
\end{align*}
where $$a:=1-2s, \quad \quad a \in (-1,1).$$

Assume $\varphi \in C^{1,s+\delta}(B_2)$, for some $\delta>0$ and $\|\varphi \|_{C^{s+\delta}}\le 1$. We extend $\varphi$ to $\mathcal B_1$ in the following way:
\begin{equation}\label{extf}
\tilde \varphi(x,y):= \varphi * \rho_{|y|},
\end{equation}
with $\rho_r(X):=r^{-n-1}\rho(X/r)$, and $\rho$ a symmetric mollifier supported in $\mathcal B_{1}$. Then it is easy to check that
$\tilde \varphi \in C^{1,s+\delta}$ is even in $y$ and is smooth away from $\{y=0\}$, and
\begin{equation}\label{lab}
\|D^2 \tilde \varphi \| \le C |y|^{s+\delta -1}  \quad \Rightarrow \quad |y|^{-a}L_a \tilde \varphi \le C |y|^{s+\delta -1}.
\end{equation}

Define, 
$$\tilde u(x,y)=u(x,y) - \tilde \varphi(x, y),$$
and let 
$\Lambda :=\{\tilde u (x,0)=0\}$. Then $\tilde u$ satisfies
\begin{equation*}\begin{cases}
 \tilde u(x,0) \geq 0  \quad \text{for $x\in B_1$}\\
 \tilde u(x,y)=\tilde u(x,-y)\\
 L_a \tilde u= - L_a \tilde \varphi \quad \text{in $\mathcal B_1 \setminus \Lambda$}
\end{cases}\end{equation*}

Denote by
$$F(r) := \frac{1}{r^{n+a}}\int_{\p \mathcal B_r} \tilde u^2 |y|^a d\sigma,$$
and notice that if for example $\tilde u$ is homogenous of degree $\sigma$, then $F(r)=c \, r^{2 \sigma}$, 
hence $ \frac 12 r \, \frac{d}{dr} \log F=  \sigma$.

\begin{thm}[Almgren's monotonicity formula]\label{monotone} Let $0 \in \Lambda$ and $ \alpha \in (s,s+\delta)$. 
There exist constants $C_0$ and $r_0$ depending on $\alpha$, $s$ $n$, and $\delta$ such that the function 
$$\Phi_{\tilde u}(r): = \frac 1 2(r+C_0 r^{1+\eps}) \frac{d}{dr} \log \left(\max\{F(r), r^{2(1+\alpha)}\} \right)$$ is monotone increasing for all $0< r \leq r_0,$ where $\eps > 0$ is small so that $s+\delta \ge \alpha +\eps.$ 
\end{thm}

For simplicity we also use the notation of the ``averages" of a function $g$ with respect to the measures $|y|^a d \sigma$ and $|y|^adX$:
$$\fint_{\p \mathcal B_r} g  \, \, |y|^a d\sigma : = \frac{1}{ r^{n+a}} \int_{\p \mathcal B_r} g |y|^ad\sigma$$
and 
$$\fint_{\mathcal B_r} g \, \, |y|^a dX : = \frac{1}{r^{n+1+a}} \int_{\mathcal B_r} g |y|^ad X.$$
With this notation,
$$F(r) := \fint_{\p \mathcal B_r} \tilde u^2 |y|^a d\sigma,$$ and 
$$F'(r) = 2 \fint_{\p \mathcal B_r} \tilde u \tilde u_\nu|y|^a d\sigma.$$

First, we prove the following preliminary lemma.

\begin{lem} \label{positive}Assume $F(r) \geq r^{2(1+\alpha)}$. Then, for $r$ small
$$\fint_{\mathcal B_r} \tilde u^2 |y|^a dX \leq C F(r).$$
$$r^{-1} F'(r) \sim \fint_{\mathcal B_r} |\nabla \tilde u|^2 |y|^a dX \geq C r^{ -2} F(r)$$
\end{lem}
\begin{proof}
Assume for simplicity that $u(0)=\varphi(0)=0$, $\nabla \varphi(0)=0$, hence $$|\tilde \varphi| \le C r^{1+s+\delta} \le r^{1+\alpha+\eps} \quad \mbox{ in $B_r$},$$ 
hence the functions $u$ and $\tilde u$ are ``the same" up to an error of $r^{1+\eps}$.
Since $F(r) \ge r^{2(1+\alpha)}$ we obtain
$$\fint_{\p \mathcal B_r} u^2 |y|^a d\sigma \quad  \sim \quad \fint_{\p \mathcal B_r} \tilde u^2 |y|^a d\sigma=F(r).$$
Since $L_a u=0$ in the set $\{ |u|> r^{1+\alpha +\eps} \}$ we may apply the mean value inequality for the $L_a$- subharmonic function $$\left( (|u| -r^{1+\alpha +\eps})^+\right)^2$$ 
and obtain that its average in $\mathcal B_r$ is bounded by its average on $\p \mathcal B_r$. This easily gives the first inequality above.

For the second inequality we have $L_a u \le 0$ and $u(0)=0$, hence the average of $u$ on $\p \mathcal B_r$ is negative. From this and the version of Poincare inequality written for $\p \mathcal B_r$ (see Lemma 2.10 in \cite{CSS}) we obtain
$$r^2 \fint_{\mathcal B_r} |\nabla u|^2 |y|^a dX \geq c \fint_{\p \mathcal B_r} (u^+)^2 |y|^a d\sigma.$$ 

Moreover, similarly to the quoted lemma, since a function $v$ in the weighted Sobolev space $W^{2,1}(\mathcal B_1, |y|^a)$ has trace in $L^2(B_1)$, we also have the following version of Poincare inequality:
$$ r^2 \fint_{ \mathcal B_r^+} |\nabla v|^2|y|^a dX \geq c \fint_{\p \mathcal B_r^+} (v-\bar v)^2 |y|^a d\sigma $$ with 
$$\bar v := \fint_{B_r} v(x,0) dx.$$
Hence, since $u \ge -r^{1+\alpha + \eps}$ on $B_r$, we deduce that $$ r^2\fint_{\mathcal B_r} |\nabla u|^2 |y|^a dX \geq c \fint_{\p \mathcal B_r} (u^-)^2 |y|^a dX - C r^
{2(1+\alpha +\eps)}. $$
Using that $\nabla \tilde u=\nabla u + O(r^{1+\alpha+\eps})$ we obtain
$$\fint_{\mathcal B_r} |\nabla \tilde u|^2 |y|^a dX \geq C r^{ -2} F(r).$$

Finally, 
$$\int_{\mathcal B_r} (\tilde u L_a \tilde u + |\nabla \tilde u|^2 |y|^{a})dX=\int_{\mathcal B_r} \text{div}(|y|^a\tilde u \nabla \tilde u)dX=\int_{\p \mathcal B_r} \tilde u \tilde 
u_\nu|y|^a d\sigma,$$
thus, since $\tilde u L_a \tilde u= -\tilde u L_a \tilde \varphi$ we have
$$\frac {1}{2r} F'(r) = \frac 1 r \fint_{\p \mathcal B_r} \tilde u \tilde u_\nu|y|^a d\sigma=    \fint_
{\mathcal B_r} (|\nabla \tilde u|^2 - |y|^{-a}\tilde u L_a \tilde \varphi)|y|^{a}dX.$$
By Cauchy-Schwartz and the property \eqref{lab} of $\tilde \varphi$ we have
\begin{align*}\left| \fint_{\mathcal B_r} \tilde u (|y|^{-a}L_a \tilde \varphi)|y|^ad\sigma\right| &\leq \left(\fint_{ \mathcal B_r} \tilde u^2 |y|^ad\sigma \right)^{1/2} \left(\fint_{ \mathcal B_r} (|y|^{-a} L_a \tilde \varphi)^2|y|^ad\sigma\right)^{1/2}\\ &\leq C r^{\alpha +\eps -1}F(r)^{1/2},\end{align*}
and we obtain the desired conclusion (using also that $F(r) \geq r^{2(1+\alpha)}$).
\end{proof}

{\it Proof of Theorem $\ref{monotone}$.}
It is enough to consider the case when $$F(r) \geq r^{2(1+\alpha)}.$$
Then,
$$\Phi_{\tilde u}(r)= \frac 1 2(r+C_0 r^{1+\eps}) \frac{F'(r)}{F(r)}.$$
We compute its logarithmic derivative and show that it is non-negative. Precisely, we look at the quantity:
$$N(r):=\frac{1}{r}+ \frac{\eps C_0 r^{\eps-1} }{1+C_0 r^{\eps}} +  \frac{F''(r)}{F'(r)} - \frac{F'(r)}{F(r)}.$$
As in Lemma \ref{positive},
\begin{equation}\label{uunu}
\int_{\p \mathcal B_r} \tilde u \tilde u_\nu|y|^a d\sigma=  \int_{\mathcal B_r} (|\nabla \tilde u|^2 + |y|^{-a}\tilde u L_a \tilde u)|y|^{a}dX.\end{equation}
Thus,
$$F''(r)= -\frac {(n+a)}{r} F'(r) + 2\fint_{\p\mathcal B_r} (|\nabla \tilde u|^2 + |y|^{-a} \tilde u L_a\tilde \varphi)|y^a| d\sigma.$$
As in [CSS] we can estimate that
\begin{align*}\fint_{\p \mathcal B_r} |\nabla \tilde u|^2 |y|^a d\sigma &= 2 \fint_{\p \mathcal B_r} (\tilde u_\nu)^2|y|^a d\sigma + \frac{n+a-1}{r} \fint_{\p \mathcal B_r} \tilde u \tilde u_\nu |y|^a d\sigma\\& - \fint_{\mathcal B_r} ((n+a-1)\tilde u - 2 X \cdot \nabla \tilde u)(|y|^{-a} L_a\tilde \varphi) |y|^adX.\end{align*}
Hence,
$$N(r)= \frac{\eps C_0 r^{\eps-1} }{1+C_0 r^{\eps}} +  \frac{4 \fint_{\p \mathcal B_r} (\tilde u_\nu)^2|y|^a d\sigma}{F'(r)} - \frac{F'(r)}{F(r)} + \frac{H(r)}{F'(r)},$$
with
\begin{align*}H(r)& = 2 \fint_{\p \mathcal B_r} \tilde u (|y|^{-a}L_a\tilde \varphi) |y|^{a}d\sigma - (n+a-1)\fint_{\mathcal B_r} \tilde u(|y|^{-a}L_a\tilde \varphi) |y|^{a}dX\\
& +4 \fint_{\mathcal B_r} (X \cdot \nabla \tilde u)( |y|^{-a}L_a\tilde \varphi) |y|^{a}dX\\&:= H_1(r) + H_2(r) + H_3(r).\end{align*}
By Cauchy-Schwartz, we conclude that (for $r$ small)
\begin{equation}\label{N}N(r) \geq \frac{\eps C_0 r^{\eps-1} }{1+C_0 r^{\eps}} +  \frac{H(r)}{F'(r)}  \geq \eps \frac{C_0}{2} r^{\eps-1} + \frac{H(r)}{F'(r)}.\end{equation}
We now estimate $H(r).$
As in Lemma \ref{positive} we use property \eqref{lab} of $\tilde \varphi$ and conclude
$$\left| \fint_{\mathcal B_r} \tilde u (|y|^{-a}L_a \tilde \varphi)|y|^adX\right| \leq C r^{\alpha + \eps -1}F(r)^{1/2},$$
and with a similar computation 
\begin{align*}\left| \fint_{\p\mathcal B_r} \tilde u (|y|^{-a}L_a \tilde \varphi)|y|^ad\sigma\right| &\leq C r^{\alpha +\eps -1}F(r)^{1/2}.\end{align*}
 In the same way,
 \begin{align*}\left |\fint_{\mathcal B_r} (X \cdot \nabla \tilde u)( |y|^{-a}L_a\tilde \varphi) |y|^{a}dX\right|& \leq r \left(\fint_{\mathcal B_r} |\nabla \tilde u|^2 |y|^{a}dX\right)^{1/2}\left(\fint_{\mathcal B_r} (|y|^{-a} L_a \tilde \varphi)^2|y|^a dX\right)^{1/2}\\ & \leq r^{\alpha + \eps}\left(\fint_{\mathcal B_r} |\nabla \tilde u|^2 |y|^{a}dX\right)^{1/2},\end{align*}
hence by Lemma \ref{positive}
$$\frac{|H_1(r)|}{F'(r)} \leq Cr^{\eps -1}, \quad \quad \frac{|H_2(r)|}{F'(r)} \leq Cr^{\eps -1}, \quad \quad \frac{|H_3(r)|}{F'(r)} \leq Cr^{\eps -1}.$$
 Combining these estimates with \eqref{N} we get that $N(r)>0$ for $C_0$ large and $r$ small.
\qed

\medskip

Now the arguments in [CSS] apply, and they give that if $0 \in \p \Lambda$ then the limit 
$\Phi(0+)$ can take only two values: $1+s$ and $1+\alpha$, and this implies the $C^{1,s}$ 
regularity of $u$. If this limit $\Phi(0+)$ equals $1+s$ we say that $0$ is a regular point. Then the monotonicity 
formula allows us to perform the blow-up analysis at a regular point and to obtain the 
$C^{1,\gamma}$ regularity of the free boundary.
In view of this, we sharpen the regularity results of [CSS] for the thin obstacle problem, in the 
case when the obstacle $\varphi \in C^{1,s+\delta},$ and obtain Theorem \ref{T31}.

\subsection{An extension of Theorem \ref{monotone}}  We consider here the case when the obstacle $\varphi$ is $C^{1+s+\delta}$ only in a certain pointwise sense and $u$ has nearly optimal regularity. This case appears in \cite{CDS} where we deal with the obstacle problem for non-local minimal surfaces. 
Precisely, we obtain the following proposition.

\begin{prop} \label{ls5} Let $u \in C^{2s+\eps}$ solve the obstacle problem \eqref{opti3}-\eqref{opti4}, $0 \in \p \Lambda$. Assume that $\|u\|_{L^1(\R^n,d\omega )} \le 1$ and $\nabla u$ is pointwise $C^{s- \frac \delta 2}$ at the origin, i.e.
\begin{equation}\label{uo}
|\nabla u(x)| \le |x|^{s-\frac \delta 2}\quad \quad \mbox{in $B_1$.}
\end{equation}
If $\varphi \in C^{2s+\eps}$, $\nabla \varphi$ is pointwise $C^{s+\delta}$ at the origin i.e., for all $r <1$
\be\label{nab_phi}
|\nabla \varphi|_{L^\infty(B_r)} \le r^{s+\delta}  \quad \mbox{if $s\in (0, \frac 12)$}\ee
$$ [\nabla \varphi]_{C^{2s + \delta -1}(B_r)} \le r^{1-s}\quad \mbox{if $s\in [ \frac 12,1)$},$$
and $f$ satisfies\be\label {conf}
[f]_{C^\gamma(B_r)} \le C r^{s+\delta} \quad \mbox{for some $\gamma>1-2s$, if $s \in (0,1/2)$},  \ee
$$[f]_{C^\delta(B_r)} \le C r ^{1-s} \quad \mbox{if $s \in [1/2,1)$,}  $$ 
then $u$ is pointwise $C^{1,s}$ at the origin i.e. 
\be\label{conu}
|u(x)| \le C |x|^{1+s} \quad \quad \mbox{in $B_{1}$},\ee
for some $C$ depending only on $n$, $s$ and $\delta$.
\end{prop}

The Proposition above will follow if we show that the monotonicity formula can be applied under these hypotheses.
%We mention that similar arguments with the ones that we provide below were used by Guillen in 
%[G] in a slightly different context.

Assume first that the right hand side $f$ equals $0$. Since $u, \varphi \in C^{2s + \eps}$ in $B_1$, the integrations by parts performed in the monotonicity formula are justified. Now, using the boundary estimates for the equation $L_a u=0$ together with $ y^a u_y(0,y) \to 0$ as $y \to 0$ which is a consequence of $0 \in \p \Lambda$, we find that the extension $u(X)$ satisfies in $\mathcal B_r$
\begin{equation}\label{ubd}
| u| \le C r^{1+s-\frac \delta 2}, \quad |X\cdot \nabla u| \le C r^{1+s-\frac \delta 2}.
\end{equation}
In view of \eqref{nab_phi}, the extension $\tilde \varphi$ defined in \eqref{extf} satisfies in $\mathcal B_r$
$$|\tilde \varphi| \le r^{s+\delta+1}, \quad |\nabla \tilde \varphi| \le r^{s+\delta},$$
and
$$\frac{|u_y|}{|y|}, |D^2 \tilde \varphi| \le C r^{s+\delta} |y|^{-1}  \quad \mbox{if $s\in (0, \frac 12)$ or},  $$
$$\frac {|u_y|}{|y|} ,|D^2 \tilde \varphi| \le r^{1-s}|y|^{2s+\delta-2} \quad \mbox{if $s\in [ \frac 12,1)$}.$$
Since $a=1-2s$ and
$$ |y|^{-a}|L_a \tilde \varphi|\le C \left (|D^2 \tilde \varphi|+ \frac {|u_y|}{|y|}\right),$$
we see that $|y|^{-a}L_a \tilde \varphi$ is integrable with respect to the measures $|y|^a dX$ and $|y|^a d \sigma$, and its averages with respect to these measures in $\mathcal B_r$, respectively $\p \mathcal B_r$ are bounded by $C r^{s+\delta-1}$.

From these inequalities we see that $\tilde u=u-\tilde \varphi$ satisfies the same bounds in \eqref{ubd} and we can estimate the error terms $H_1$, $H_2$, $H_3$ by $$C r^{1+s-\frac \delta 2}r^{s+\delta-1}=Cr^{2s+\delta/2} \le C r^{2 \alpha +\varepsilon},$$ provided that $\alpha$ is taken sufficiently close to $s$ and $\varepsilon>0$ is small. The difference is that now we used the $L^\infty L^1$ bound for the product between the $\tilde u$ terms and $|y|^{-a}L_a \tilde \varphi$ terms instead of the $L^2L^2$ as before.  

In the general case when the right hand side $f$ is not $0$, then the potential whose fractional Laplacian equals $f$ must satisfy \eqref{nab_phi} and we need to impose the conditions in \eqref{conf}.

We mention that similar arguments with the ones that we provide above were used by Guillen in 
\cite{G} in a slightly different context.

\section{Appendix}

Below we discuss the Schauder estimates for translation invariant integro-differential equations of the type 
$$\mathcal L_K v(x)= P.V.\int (v(x+y)-v(x))K(y) dy,$$
with kernels $K$ that satisfy
\begin{equation}\label{Kca}
 \frac{\lambda}{|y|^{n+2s}} \leq K(y) \leq  \frac{\Lambda}{|y|^{n+2s}}, \quad 0<\lambda \leq \Lambda, 
\end{equation}
\begin{equation}\label{Kc1a}
|\nabla K(y)| \le \Lambda |y|^{-(n+1+2s)}.
\end{equation}
For convenience we state again the Schauder estimates used in Section 5.
 
\begin{prop}\label{SE}
Let $K$ be a symmetric kernel that satisfies \eqref{Kca}, and assume that $v \in L^1(\R^n,d\omega)$ satisfies
$$\mathcal L _K v=f \quad \mbox{in $B_1$}, \quad \quad \quad \|v\|_{L^\infty(B_1)} \le 1.$$

a) If $\|f\|_{L^\infty (B_1)} \le 1 $, $\|v\|_{L^1(\R^n,d \omega)} \le 1$ then $$\|v\|_{C^\alpha(B_{1/2})} \le C(\alpha), \quad \mbox{for any $\alpha < 2s$.}$$

b) If $K$ satisfies \eqref{Kc1a} and 
$$ \int_{\mathcal C B_1} v  \, \, |x|^{-(n+2s+1)} \, dx \le 1, \quad [f]_{C^\gamma (B_1)} \le 1, \quad \mbox{ for some $\gamma \in (0,1)$}$$
 then $$\|v\|_{C^{2s+\gamma}(B_{1/2})} \le C(\gamma),$$
provided that $2s+\gamma$ is not an integer.

c) Conversely, if $K$ satisfies \eqref{Kc1a} and $\|v\|_{L^1(\R^n,d \omega)} \le 1$, $\|v\|_{C^{2s+\gamma}(B_1)} \le 1$, then $$\|f\|_{C^\gamma(B_{1/2})} \le C.$$
\end{prop}

We remark that the constant $C(\gamma)$ in part b) is independent on $\|f\|_{L^\infty}$ and $\|v\|_{L^1(\R^n,d \omega)}$.

We point out that by the results in \cite{Ser}, one could in fact relax the assumption \eqref{Kc1a} and require that it is satisfied only outside of a neighborhood of the origin.

We sketch the main steps in the proofs of parts a) and b) and use similar ideas as in Section 5. The proof of part c) is standard and we do not include it here.
 
First we obtain a Liouville type result for global solutions which have integrable decay at infinity.

\begin{lem}\label{al1}
The only global solutions to the equation
$$\mathcal L_{K} v = 0 \quad \mbox{in $\R^n$}, \quad \|v\|_{L^\infty(B_{R_k})} \le R_k^\alpha, \quad \mbox{with $R_k=2^k$, $k \ge 0$,}$$ 
for some $\alpha<2s$, are constant if $s \le \frac 12$, or linear if $s \in (\frac12,1)$
\end{lem}

\begin{proof}
Since $\alpha<2s$ we can apply the H\"older estimates from \cite{S2} (as in Section 5) and we obtain that 
\begin{equation}\label{a1}
 \| v \|_{C^{\eps_0}(B_{1/2})} \le C ,
 \end{equation}
for some $C$, $\eps_0$ depending only on $n$, $s$, $\alpha$. 
Since the function $R_k^{-\alpha}v(R_k x)$ satisfies the same hypotheses as $v$, we can apply the estimate above for this function and obtain
\begin{equation}\label{a2}
\|v\|_{C^{\eps_0}(B_{R_k/2})} \le C R_k^{\alpha-\eps_0}.
\end{equation}
This means that the discrete difference function
$$\tilde v:=\frac{1}{C_0} \frac{u(x+he)-u(x)}{h^{\eps_0}},\quad \quad |e|=1, h\in [0,1],$$
also satisfies the hypotheses of $v$ with $\alpha$ replaced by $\alpha -\eps_0$.

We apply the estimates \eqref{a2} for $\tilde v$ and we obtain (see Lemma 5.6 in \cite{CC})
$$\|v\|_{C^{2\eps_0}(B_{R_k})} \le C R_k^{\alpha - 2\eps_0}.   $$

We iterate this result and distinguish 2 cases, if $\alpha<1$ or $\alpha \ge 1$.

If $\alpha<1$ then we find $$\|v\|_{C^{\alpha'}(B_{R/2})} \le C R^{\alpha - \alpha'},   $$ 
for some $\alpha' \in (\alpha,1)$ and by letting $ R \to \infty$ we obtain that $v$ is a constant.

If $\alpha \ge 1$ then we obtain
$$\|v\|_{C^{0,1}(B_{R_k})} \le C R_k^{\alpha - 1},$$
hence the discrete difference quotient $(v(x+he)-v(x))/h$ satisfies the hypotheses of the lemma with exponent $\alpha-1 <1$ thus it must be constant, which gives that $v$ is a linear function.
\end{proof}

Using compactness and Lemma \ref{al1} we obtain the following interior estimate.

\begin{lem}\label{al2}
 Let $w$ be a solution to the truncated kernel equation
\begin{equation}\label{a3}
\mathcal L_{K_T} w = g \quad \mbox{in $B_{1/2}$}, \quad K_T:=\chi_{B_{1/2}} K,
\end{equation}
$$ \|g\|_{L^\infty(B_{1/2})}\le 1, \quad \quad \|w\|_{L^\infty(B_1)} \le 1.$$
Then, for any $\alpha<2s$ we have
$$\|w\|_{C^\alpha(B_{1/4})} \le C(\alpha).$$
\end{lem}

\begin{proof}
We may assume that $\alpha \ne 1$. We need to show that if $w$ satisfies 
\begin{equation}\label{wk}|w-l_k| \le r_k^\alpha \quad \mbox{in $B_{r_k}$}, \quad r_k=2^{-k},\end{equation}
for $k=0,1,...,m$ for some $m \ge k_0$ sufficiently large, then the inequality above holds also for $k=m+1$. Here $l_k$ is either a constant (for $\alpha<1$) or a linear function (for $\alpha>1$).
Indeed, as $k_0 \to \infty$, we may find a subsequence of rescalings
$$\tilde w:=r^{-\alpha} w(r x)  \quad \quad r=r_m$$ 
which converges uniformly on compact sets to a function $v$ that satisfies the hypotheses of Lemma \ref{al1}, and then \eqref{wk} is clearly verified for $k$ large. The uniform convergence on compact sets is once more guaranteed by Harnack inequality since $\tilde w$ satisfies
$$\mathcal L_{\tilde K_T} \tilde w = \tilde g(x):=r^{2s-\alpha}g(rx), \quad \quad \tilde K_T=\tilde K \chi_{B_{r^{-1}/2}},$$
and, as $k_0 \to \infty$, we have $\tilde g \to 0$ uniformly on compact sets.
\end{proof}

The estimate in part a) of Proposition \ref{SE} follows from Lemma \ref{al2}. We write the original equation in terms of the truncated kernel $K_T$ and obtain
$$\mathcal L_{K_T} v(x)= f(x) - h(x) \quad \quad \mbox{in $B_{1/2}$},$$
with $$h(x)= \int_{\mathcal C B_{1/2}} (v(x+y)-v(x))K(y)dy,$$
and clearly $$ |h(x)| \le C ( \|v\|_{L^1(d \omega)} + |v(x)|) \le C.$$

\

Next we apply Lemma \ref{al2} for difference quotients and obtain the $C^{2s+\gamma}, \gamma \in (0,1),$ estimate.

\begin{lem}\label{al3}
Assume that $K$ satisfies \eqref{Kc1} (only outside a neighborhood of the origin) and $w$ satisfies 
$$\mathcal L_{K_T} w = g + a \, \, w \quad \mbox{in $B_{1/2}$}, \quad \quad \quad \|w\|_{L^\infty(B_1)} \le 1,$$
for some constant $a$ with $|a| \le 1$, and with 
\begin{equation}\label{a4}
\|g\|_{C^{0,1}(B_{1/2})} \le 1.
\end{equation} 
Then, if $\alpha<2s$ we have
$$\|w\|_{C^{1+\alpha}(B_{1/4})} \le C(\alpha).$$
\end{lem}

\begin{proof}
Since the right hand side is bounded, we obtain by Lemma \ref{al2} a $C^{\alpha_0}$ bound for $w$ in $B_{1/4}$ for some $\alpha_0 \in (0,2s)$. Then we iterate Lemma \ref{al2} a finite number of times for the discrete differences of $w$ and 
successively estimate $w$ in $C^{\alpha_k}(B_{r_k})$ with $r_k=4^{-k}$ and $\alpha_0<\alpha_1<\alpha_2<...< \alpha_m=1$. Then we iterate this argument one more time and obtain the desired conclusion.

Notice that in order to apply Lemma \ref{al2} in $B_{r_k}$ instead of $B_1$ we need to write the equation for the truncated kernel $$K_{T,k}:= K_T \chi_{B_{r_k/2}}.$$
Then the right hand side gets modified as follows
$$\mathcal L_{K_{T,k}} w(x)= g(x) - h_1(x) +h_2(x)   $$
with $$h_1(x)= \int_{ B_{1/2}\setminus B_{r_k/2}} w(x+y)K(y)dy= \int_{B_{1/2}(x)\setminus B_{r_k/2}(x)} w(y)K(x-y)dy,$$
$$ h_2(x)= a w(x) + \int_{B_{1/2}\setminus B_{r_k/2}} w(x) K(y) dy= (a+C(K) ) w(x).$$
 
From our hypothesis on $K$, arguing as in Step 2 of Proposition \ref{po1}, we find $\|h_1\|_{C^{0,1}} \le C$. Since $\|h_2\|_{C^{\alpha_k}} \le C \|w\|_{C^{\alpha_k}}$ in $B_{r_k/2}$, we can apply Lemma \ref{al2} for the discrete difference
$$\frac{w(x+he)-w(x)}{h^{\alpha_k}},$$
and obtain the $C^{\alpha_{k+1}}$ bound for $w$ in $B_{r_k/4}$. 
\end{proof}

Finally we prove part b) of Proposition \ref{SE}.

\begin{lem}\label{al4}
Assume that $v$ satisfies the hypotheses of part b) in Proposition \ref{SE} with 
$$\|v\|_{L^\infty(B_1)} \le \delta_0, \quad \quad [f]_{C^\gamma (B_1)} \le \delta_0$$
for some small $\delta_0$.
Then there exist  polynomials $p_k$ of degree $[\beta ]$, and $p_0\equiv 0$, such that 
$$|v-p_k| \le r_k^{\beta}  \quad \mbox{in $B_{r_k}$} , \quad r_k=2^{-k}, \quad \quad \beta:=2s+\gamma,$$
for all $k \ge 0$.
\end{lem}

\begin{proof}

We prove the lemma by induction by showing that if the conclusion holds up to some $k$ large, then it holds also for $k+ m_0$ for some fixed $m_0$.

By the induction hypothesis, the coefficients of the polynomials $p_k$ are uniformly bounded. Hence, if $\psi$ is a cutoff function which is $1$ in $B_{1/2}$ and $0$ outside $B_1$, then $p_k\psi$ is a $C_0^\infty$ function with a uniform $L^\infty$ bound and 
$$\mathcal L_K(p_k \psi)=q \quad \mbox{with} \quad \|q\|_{C^{0,1}}\le C.$$
 
Now we write the equation for the rescaling $\tilde v$ of $v-p_k \psi$
$$\tilde v (x) = r^{-\beta}(v-p_k \psi)(rx), \quad \quad r=r_k,$$
and obtain
$$\mathcal L_{\tilde K} \tilde v (x) =\tilde g(x):=r^{-\gamma}f(rx) + r^{-\gamma} q(rx) \quad \quad \mbox{in $B_{r^{-1}}$.}  $$
Notice that $[\tilde g]_{C^\gamma} \le C \delta_0$ in $B_2$ provided that $r$ is sufficiently small, and by the induction hypothesis
$$|\tilde v|_{L^\infty(B_1)}\le 1, \quad |\tilde v(x)| \le C |x|^\beta  \quad \mbox{in} \quad B_{r^{-1}}\setminus B_1,$$
which gives
\begin{equation}\label{final}\int_{\mathcal C B_1} |\tilde v|  \, \, |x|^{-(n+2s+1)} \, dx \le C_0,\end{equation}
for a fixed $C_0$ depending only on $\gamma$ and the universal constants.

As in Lemma \ref{al3} we write the equation for $\tilde v$ in $B_{1/2}$ using the truncated kernel $ \tilde K_T$ and obtain
$$\mathcal L_{\tilde K_T} \tilde v= \tilde g  - h + C(K) \tilde v=: g_0 + a \, \,  \tilde v,$$
with
$$h(x)= \int_{\mathcal C B_{1/2}(x)} \tilde v(y)\tilde K(x-y)dy.$$
From the hypothesis on $K$ and \eqref{final} we find $$[h]_{C^{0,1}(B_2)} \le C_1.$$
We use the estimate on the $C^\gamma$ seminorm of $g_0$ and deduce that 
\begin{equation}\label{a5}
\|g_0\|_{C^\gamma(B_{1/2})} \le C,
\end{equation} 
by obtaining an $L^\infty$ bound for $g_0$. We achieve this by sliding the paraboloid $4 |x|^2$ by above till it touches the graph of $ \tilde v$ at some $x_0 \in B_1$. Then $\mathcal L_{K_T}v(x_0) \le C$ hence $g_0(x_0) \le C$, and similarly we find a point $x_1$ such that $g_0(x_1) \ge -C$, and this proves \eqref{a5}.  

By Lemma \ref{al2} the function $\tilde v$ is uniformly H\"older continuous in $B_1$. Moreover, $g_0$ is the sum of a Lipschitz function (with bounded Lipschitz norm) and a function with $C^\gamma$ norm bounded by $C \delta_0$. By compactness and Lemma \ref{al3} we find that as $\delta_0 \to 0$ we can approximate $\tilde v$ uniformly in $B_{1}$ by a function with bounded $C^{1+\alpha}$ norm in $B_{1/4}$ (with $1+2s>1+\alpha > \beta$.) Thus we can find $m_0$ universal such that
$$|\tilde v - \tilde p| \le \rho^\beta \quad \mbox{in $B_\rho$}, \quad \rho=2^{-m_0}.$$
This means that the induction hypothesis holds for $k+ m_0$, and the lemma is proved.  
\end{proof}

\end{document}